\documentclass[12pt]{amsart}
\usepackage{amsmath}
\usepackage{amsxtra}
\usepackage{amstext}
\usepackage{amssymb}
\usepackage{amsthm}
\usepackage{latexsym}
\usepackage{dsfont} 
\usepackage{verbatim}
\usepackage{tabls}
\usepackage{graphicx}
\usepackage{rotating}
\usepackage{caption}
\theoremstyle{plain}
\newtheorem{thm}{Theorem}[section]
\newtheorem{cor}[thm]{Corollary}
\newtheorem{lem}[thm]{Lemma}
\newtheorem{prop}[thm]{Proposition}
\newtheorem{ques}[thm]{Question}

\theoremstyle{definition}
\newtheorem{defn}[thm]{Definition}

\newtheorem{rem}[thm]{Remark}

\newtheorem{cla}[thm]{Claim}

\numberwithin{equation}{section}
\newcommand{\ci}[1]{_{{}_{\scriptstyle{#1}}}}

\def\T1{\overline{T}}
\def\dist{\operatorname{dist}}
\def\supp{\operatorname{supp}}
\def\Lip{\operatorname{Lip}}
\def\eps{\varepsilon}

\def\diam{\operatorname{diam}}

\def\R{\mathbb{R}}

\def\XXint#1#2#3{{\setbox0=\hbox{$#1{#2#3}{\int}$}
     \vcenter{\hbox{$#2#3$}}\kern-.5\wd0}}

\begin{document}

\title[Reflectionless measures II]
{Reflectionless measures for Calder\'{o}n-Zygmund operators II:\\
Wolff potentials and rectifiability.}

\author[B. Jaye]
{Benjamin Jaye}
\address{Department of Mathematical Sciences,
Kent State University,
Kent, OH 44240, USA}
\email{bjaye@kent.edu}

\author[F. Nazarov]
{Fedor Nazarov}
\email{nazarov@math.kent.edu}

\date{\today}

\maketitle

\section{Introduction}

We continue our study of the reflectionless measures associated to an $s$-dimensional Calder\'{o}n-Zygmund operator (CZO) acting in $\mathbb{R}^d$ with $s\in (0,d)$. Here, our focus will be the study of CZOs that are \textit{rigid}, in the sense that they have few reflectionless measures associated to them.  Our goal is to prove that the rigidity properties of a CZO $T$ impose strong geometric conditions upon the support of any measure $\mu$ for which $T$ is a bounded operator in $L^2(\mu)$.  In this way, we shall reduce certain well-known problems at the interface of harmonic analysis and geometric measure theory to a description of reflectionless measures of singular integral operators.  What's more, we show that this approach yields promising new results.

Our rigidity results split into two cases, depending on whether the dimension of the CZO is integer or not.

\subsection{Ahlfors-David rigidity.}  We begin by describing a rigidity result that appeared in our previous paper \cite{JN1}.  In \cite{JN1}, we gave a new proof of the Mattila-Melnikov-Verdera theorem \cite{MMV}, which states that the support of an Ahlfors-David (AD) regular measure $\mu$ for which the associated Cauchy transform operator is bounded in $L^2(\mu)$ is uniformly rectifiable.    The key element of the proof was showing that the Cauchy transform is AD-\textit{rigid} in the sense that the only AD-regular reflectionless measures associated to it are of the form $c\mathcal{H}^1|L$ for a line $L$, and a constant $c>0$. \\

Amongst the results we want to present here is a generalization of this idea to general integer dimensional CZOs acting in $\mathbb{R}^d$. Fix $s\in \mathbb{Z}$.  We call an $s$-dimensional CZO $T$ AD-\textit{rigid} if every AD-regular reflectionless measure associated to it takes the form $c\mathcal{H}^s|L$ for an $s$-plane $L$.  (On the other hand, every measure of this form is reflectionless for any $s$-dimensional CZO $T$, and so this rigidity condition is that a CZO $T$ should have as few AD-regular reflectionless measures associated to it as possible.)

 We shall show that \textit{if $T$ is AD-rigid, and $\mu$ is an AD-regular measure for which $T$ is bounded in $L^2(\mu)$, then $\mu$ is $s$-uniformly rectifiable}.  More precisely, we shall show that one of the geometric criteria for $s$-uniform rectifiability given by David and Semmes \cite{DS} is satisfied under the rigidity assumption, see Proposition \ref{intcase}.

\subsection{Wolff rigidity.}  Fix $s\not\in \mathbb{Z}$.  A theorem of Vihtill\"{a} \cite{Vih} states that the $s$-dimensional Riesz transform, the CZO with kernel $K(x) = \tfrac{x}{|x|^{s+1}}, x\in \mathbb{R}^d$, cannot be bounded in $L^2(\mu)$ if $\mu$ has positive lower density on a set of positive $\mu$-measure, i.e.,
$$\mu\Bigl(\Bigl\{x\in \R^d:\liminf_{r\rightarrow 0}\frac{\mu(B(x,r))}{r^s}>0\Bigl\}\Bigl)>0.$$

Because of this, the condition of AD-regularity is too strong to develop an interesting theory of measures with bounded non-integer dimensional CZOs.  We shall therefore remove the lower bound condition in the definition of AD-regularity, and consider measures $\mu$ satisfying the growth condition \begin{equation}\label{growth}\mu(B(x,r))\leq r^s, \text{ for any }x\in \mathbb{R}^d \text{ and }r>0.\end{equation}

\begin{rem}\label{growthnecessary}  For a wide class of non-degenerate CZOs including the $s$-Riesz transform, the condition (\ref{growth}) is in fact a necessary condition for the CZO associated to a non-atomic measure $\mu$ to be bounded in $L^2(\mu)$, see for instance \cite{Dav1}.
\end{rem}

The question of most interest for non-integer dimensional CZOs is to find the correct quantitative version of Vihtill\"{a}'s theorem.  Following Mateu, Prat and Verdera \cite{MPV}, we introduce the Wolff potential of a measure.  For $p\in (0,\infty)$, the $p$-Wolff potential of $\mu$ is defined by\footnote{In the potential theory literature (e.g. \cite{AH}), our $p$-Wolff potential of $\mu$ would be denoted by $\mathbb{W}_{\tfrac{p(d-s)}{p+1},\tfrac{p+1}{p}}(\mu)$.},
\begin{equation}\label{powerwolff}\mathbb{W}_p(\mu)(x) = \int_0^{\infty}\Bigl(\frac{\mu(B(x,r))}{r^s}\Bigl)^p \frac{dr}{r}.
\end{equation}
The Mateu-Prat-Verdera criterion states that \textit{if $s\in (0,d)$ and $\mu$ is a measure that satisfies the condition
\begin{equation}\label{MPVcond}\int_{Q}\mathbb{W}_2(\chi_{Q}\mu)d\mu\leq \mu(Q) \text{ for every cube }Q\subset \R^d,
\end{equation}
then every CZO $T$ associated to $\mu$ is bounded in $L^2(\mu)$}.  We include a proof of this fact in Appendix \ref{squarewolffcond} for the benefit of the reader, as it is not readily found in the literature in the generality stated here.

\medskip

We are interested in the extent to which conditions such as the Mateu-Prat-Verdera condition (\ref{MPVcond}) are necessary for the $L^2(\mu)$ boundedness of a particular CZO $T$.

We declare that an $s$-dimensional CZO $T$ is \textit{Wolff}-\textit{rigid} if the only reflectionless measure associated to it satisfying the growth condition (\ref{growth}) is the zero measure.
\medskip

In Proposition \ref{thm1} below, we shall show that \emph{if a CZO $T$ is Wolff-rigid, then there exists $p\in (0,\infty)$, depending on $s$, $d$, and the regularity of the kernel of $T$, such that for every measure $\mu$ satisfying (\ref{growth}) for which $T$ is bounded in $L^2(\mu)$,  we have
$$\int_{B(x,r)}\mathbb{W}_p(\chi_{Q}\mu)d\mu\leq C\mu(Q) \text{ for every cube }Q\subset \R^d,
$$
where $C>0$ depends on $s$, $d$, and the operator norm of $T$.}

\subsection{The Riesz transform}

Our interest in Propositions \ref{intcase} and \ref{thm1} comes from certain well known questions regarding  the $s$-Riesz transform, the CZO with kernel $K(x) = \tfrac{x}{|x|^{s+1}}, x\in \mathbb{R}^d$.  Throughout this section, $\mu$ will denote a measure for which the associated $s$-Riesz transform operator bounded in $L^2(\mu)$.

David and Semmes \cite{DS} asked whether, in the case when $s\in \mathbb{Z}$ and $\mu$ is AD-regular measure, $\mu$ is  $s$-uniformly rectifiable.  This was settled when $s=1$ by Mattila, Melnikov, and Verdera \cite{MMV}, and when $s=d-1$ by Nazarov, Tolsa and Volberg \cite{NToV12}.  At the same time as \cite{NToV12}, a series of papers by Hofmann, Martel, Mayboroda and Uriate-Tuero \cite{HM, HMM, HMU} proved the result under an additional hypothesis on the support of $\mu$.  The cases $s=2, \dots, d-2$ remain open.
\medskip

Regarding the non-integer dimensional case, Mateu, Prat and Verdera \cite{MPV} proved that if $s\in (0,1)$ then $\mu$ satisfies (\ref{MPVcond}).  Thus, if $s\in (0,1)$, then the $s$-Riesz transform associated to $\mu$ is bounded in $L^2(\mu)$ if and only if (\ref{MPVcond}) holds.

This result is rather surprising due to the fact that the Riesz kernel is a sign-changing vector field whereas the Wolff potential has a positive kernel.  In particular, the estimate implies that the Calder\'{o}n-Zygmund capacity defined by the Riesz transform is equivalent to a certain positive non-linear capacity from potential theory.  

In \cite{MPV} (and elsewhere, for instance \cite{ENV, Tol11}) it was conjectured that (\ref{MPVcond}) should hold true for $s>1$, $s\not\in \mathbb{Z}$.  This conjecture is open for $s>1$\footnote{While this paper was in preparation, Reguera, Tolsa and the two authors proved this conjecture for $s\in (d-1,d)$ by combining the ideas developed here with those developed in the paper \cite{RT} mentioned below.}.  For $s\in (d-1,d)$, Eiderman, Nazarov and Volberg \cite{ENV} showed that the support of $\mu$ cannot have finite $s$-dimensional Hausdorff measure.  This is a qualitative version of the condition (\ref{MPVcond}).  For $s\in (1,d-1)$, $s\not\in \mathbb{Z}$ even showing that this qualitative property holds remains an open problem.


The sharp estimate (\ref{MPVcond}) was recently verified for $s\in (d-1,d)$ for measures supported on uniformly disconnected sets by Reguera and Tolsa \cite{RT}.   The problem is understood for all $s\in (0,d)$ if the measure is precisely the $s$-dimensional Hausdorff measure restricted to a Cantor-type set  \cite{Tol11,EV}.

In short, the results that are known for general measures split into two cases:   $s\in (0,1]$ and $s\in [d-1,d)$.  In the first case, the powerful Menger curvature formula, first introduced to the area by Melnikov, is available.  In the latter case, one can make use of a strong maximum principle for the operator $(-\Delta)^{\alpha}$ when $\alpha\leq 1$.

The main challenge is to come up with techniques that can apply to intermediate cases in which neither the Menger curvature formula, nor the strong maximum principle, is readily available.  It is our hope that reflectionless measures may provide such a tool.  As such, we pose the following question regarding the rigidity of the $s$-Riesz transform.

\begin{ques}\label{reflconj}  Is the $s$-Riez transform sufficiently rigid?  In other words, suppose that $\mu$ is a reflectionless measure for the $s$-Riesz transform satisfying (\ref{growth}).

\indent (a).   If $s\not\in \mathbb{Z}$, then is $\mu$ necessarily the zero measure?

\indent (b).  If $s\in \mathbb{Z}$, and $\mu$ is $s$-AD regular, then does $\mu$ coincide with a constant multiple of the $s$-dimensional Hausdorff measure restricted to an $s$-plane?
\end{ques}


Out of the two parts of this question, we are more confident that part (a) of the question should be correct as stated, and in this paper we verify that this is the case when $s\in (d-1,d)$ (Proposition \ref{onlytrivmeas}).  Combining this with the non-integer rigidity result mentioned above (Proposition \ref{thm1}) yields the following theorem.

\begin{thm}  Let $s\in (d-1,d)$.  There exists $p\in (0,\infty)$, depending on $s$, $d$, such that if $\mu$ is a non-atomic measure for which the associated $s$-Riesz transform operator is bounded in $L^2(\mu)$, then
$$\int_{Q}\mathbb{W}_{p}(\chi_{Q}\mu)d\mu\leq C\mu(Q), \text{ for every cube }Q\subset \R^d,
$$
for a constant $C>0$ depending on $s$, $d$, and the operator norm of the Riesz transform.
\end{thm}

As we have already mentioned, in a subsequent paper written in collaboration with Reguera and Tolsa, the sharp exponent $p=2$ is proved.  We would like to emphasize that the proof in that paper builds upon, and so does not supersede, what is done here.


We are a long way from answering Question \ref{reflconj} either positively or negatively, but  at least we can say that reflectionless measures for the Riesz transform have some special structure.  More precisely, we show that for any $s\in (0,d)$, a reflectionless measure for the $s$-Riesz transform satisfying (\ref{growth}) has

\begin{itemize}
\item nowhere dense support (Section \ref{porous}), and
\item infinite energy in the sense that $\iint_{\R^d\times\R^d}\frac{1}{|x-y|^{s-1}}d\mu(x)d\mu(y)=\infty$ (Section \ref{behaviourinfinity}).
\end{itemize}

Neither property is true for a general CZO.  For instance, the two dimensional Lebesgue measure on the unit disc is a reflectionless measure for the 1-dimensional CZO with kernel $\tfrac{\overline{z}}{z^2}$ in $\mathbb{C}$, see \cite{JN2} or Part III of this series.


\section{Preliminaries}  

\subsection{General notation}
\begin{itemize}
\item By a measure, we shall always mean a non-negative locally finite Borel measure.  For a measure $\mu$, $\supp(\mu)$ denotes its closed support.  The $d$-dimensional Lebesgue measure is denoted by $m_d$.
\item For $\Lambda >0$, we say that a measure $\mu$ is  $\Lambda$-\textit{nice} if $\mu(B(x,r))\leq \Lambda r^s$ for every ball $B(x,r)\subset \mathbb{R}^d$.  A measure is nice if it is $\Lambda$-nice for some $\Lambda>0$.\footnote{Of course, we could always renormaize a measure so that if it is nice, then the growth condition (\ref{growth}) holds, but we will be renormalizing measures in a variety of ways, and so it makes some sense to keep track of this parameter.}
\item A measure $\mu$ is called $\Lambda$-\textit{Ahlfors-David (AD)-regular} if it is $\Lambda$-nice, and also $\mu(B(x,r))\geq \tfrac{1}{\Lambda}r^s$ for every $x\in \supp(\mu)$ and $r>0$. A measure is AD-regular if it is $\Lambda$-AD regular for some $\Lambda>0$.
\item For two scalar (complex) valued functions $f,g\in L^2(\mu)$, we define
$$\langle f,g\rangle_{\mu} = \int_{\mathbb{R}^d}  f  g \,d\mu,
$$
In the event that one of the two functions (say $f$) is $\mathbb{C}^{d'}$ valued, we shall write $\langle f,g\rangle_{\mu}$ to mean the vector with components $\langle f_j,g\rangle_{\mu} $, where $f_j$ are the components of $f$.
\item A function $f$ (either scalar or vector valued) is called Lipschitz continuous if
$$\|f\|_{\Lip} = \sup_{x,y\in \mathbb{R}^d,\; x\neq y} \frac{|f(x)-f(y)|}{|x-y|}<\infty.
$$
\item For an open set $U\subset \mathbb{R}^d$, $\Lip_0(U)$ denotes the set of Lipschitz continuous functions that are compactly supported in $U$.
\item We denote by $\mathcal{D}$ a lattice of \textit{open dyadic cubes} in $\mathbb{R}^d$.  (Our approach involves several limiting operations in which lattices will be shifted and rescaled, so we shall always be dealing with some dyadic lattice, rather than the standard one.)
\item  We introduce a graph structure $\Gamma(\mathcal{D})$ on a dyadic lattice $\mathcal{D}$ by connecting each dyadic cube with an edge to its children, and all neighbouring cubes of the same sidelength.  The graph distance on $\mathcal{D}$, denoted by $d(Q, Q')$, is the shortest path from $Q$ to $Q'$ in the graph $\Gamma(\mathcal{D})$.  This graph has vertex degree bounded by $2^d+2d+1$.
\item The \emph{density} of a measure $\mu$ at a cube $Q$ (not necessary dyadic) is denoted by $$D_{\mu}(Q)=\frac{\mu(Q)}{\ell(Q)^s}.$$We shall just write $D(Q)$ if the underlying measure is clear from the context.
\end{itemize}



\subsection{$s$-dimensional Calder\'{o}n-Zygmund operators}\label{CZOSec}  We recall that an $s$-dimensional CZ kernel is a function $K:\mathbb{R}^d\backslash\{0\}\rightarrow \mathbb{C}^{d'}$ satisfying

(i)  $|K(x)|\leq \tfrac{1}{|x|^s}$, for $x\in \mathbb{R}^d\backslash\{0\}$.

(ii) $K(-x) = -K(x)$ for $x\in \mathbb{R}^d\backslash\{0\}$, \text{ and }

(iii)  For some $\alpha\in (0,1]$, the function $x\rightarrow |x|^{s+\alpha}K(x)$ is H\"{o}lder continuous of order $\alpha$.

Throughout the paper, we shall be interested in homogeneous CZ kernels, and so we impose the following addition condition.

(iv)  $K(\lambda x) = \lambda^{-s}K(x)$ for $\lambda>0$.

Fix a CZ-kernel $K$.  For $\delta>0$, the regularized CZ kernel is defined by
$$K_{\delta}(x) = K(x)\Bigl(\frac{|x|}{\max(|x|,\delta)}\Bigl)^{s+\alpha}, \,x\in \mathbb{R}^d\backslash\{0\},
$$
and $K_{\delta}(0)=0$.

Notice that $|K_{\delta}(x)|\leq \tfrac{1}{\delta^s}$ for all $x\in \mathbb{R}^d$.

If $\mu$ is a $\Lambda$-nice measure, then the Cauchy-Schwarz inequality ensures that the regularized CZO transform
$$T_{\mu,\delta}(f)(x) = \int_{\mathbb{R}^d} K_{\delta}(x-y)f(y) d\mu(y), \; x\in \mathbb{R}^d,
$$
is uniformly bounded pointwise in absolute value in terms of $\delta$, $\Lambda$, and $\|f\|_{L^2(\mu)}$.  In particular, $T_{\mu,\delta}:L^2(\mu)\rightarrow L^2_{\text{loc}}(\mu)$ for any $\delta>0$.

We say that nice measure $\mu$ has associated CZO $T$ bounded in $L^2(\mu)$ if
\begin{equation}\label{opnorm}\|T\|_{\mu}:=\sup_{\delta>0}\| T_{\mu,\delta}\|_{L^2(\mu)\rightarrow L^2(\mu)}<\infty.
\end{equation}


\subsection{Reflectionless Measures}\label{reflmeas}

We briefly recall the definition of a reflectionless measure.  A more thorough description is given in Section 3 of Part I.

Fix a CZO $T$. We recall that a measure $\mu$ is said to be \textit{diffuse} if the function $(x,y)\rightarrow \tfrac{1}{|x-y|^{s-1}}\in L^1_{\text{loc}}(\mu\times\mu)$.  For a diffuse measure $\mu$, and for $f,\varphi \in \Lip_0(\mathbb{R}^d)$, we may define
$$\langle T(f\mu),\varphi \rangle_{\mu}=  \iint_{\mathbb{R}^d\times\mathbb{R}^d}K(x-y)H_{f,\varphi}(x,y)d\mu(x)d\mu(y),
$$
where
$$H_{f,\varphi} = \frac{1}{2}\bigl[f(y)\varphi(x) - \varphi(y) f(x)\bigl].
$$
If in addition $\mu$ has restricted growth at infinity, in the sense that $\int_{|x|\geq 1}\tfrac{1}{|x|^{s+\alpha}}d\mu(x)<\infty$, then we may define the pairing $\langle T(f\mu),\varphi \rangle_{\mu}$ when $f \in \Lip_0(\mathbb{R}^d)$ satisfies $\int_{\mathbb{R}^d}f\,d\mu=0$, and $\varphi$ is merely a bounded Lipschitz function.  To do this, fix $\psi\in \Lip_0(\mathbb{R}^d)$ that is identically equal to $1$ on a neighbourhood of the support of $f$, and set
$$\langle T(f\mu),\varphi \rangle_{\mu} = \langle T(f\mu),\psi\varphi \rangle_{\mu}+ \int_{\mathbb{R}^d}T(f\mu)(x)[1-\psi(x)]\varphi(x) \,d\mu(x).$$
The mean zero property of $f$ ensures that $|T(f\mu)(x)|\leq \tfrac{C_{f,\psi}}{(1+|x|)^{s+\alpha}}$ for $x\in \supp(1-\psi)$, from which the restricted growth at infinity ensures that the integral in the second term converges absolutely.\\

We say that a diffuse measure $\mu$ with restricted growth at infinity is \textit{reflectionless} for $T$ if
$$\langle T(f\mu),1 \rangle_{\mu}=0\text{ for every }f \in \Lip_0(\mathbb{R}^d)\text{ satisfying } \int_{\mathbb{R}^d}f\,d\mu=0.
$$

\subsection{The linear operator $T_{\mu}$}

Suppose that $\mu$ is a $\Lambda$-nice measure for which the associated CZO $T$ is bounded in $L^2(\mu)$.

For $f,\varphi \in \Lip_0(\mathbb{R}^d)$, we have that satisfies $H_{f,\varphi}\in \Lip_0(\mathbb{R}^d\times\mathbb{R}^d)$, and $H_{f,\varphi}(x,x)=0.$  Thus, for any $\delta>0$,
$$|K_{\delta}(x-y)H_{f,\varphi}(x,y)|\leq\frac{C(f,\varphi)}{|x-y|^{s-1}}\chi_{S\times S}(x,y),
$$
where $S\supset \supp(f)\times\supp(\varphi)$.  Now, note that $K_{\delta}(x-y)H_{f,\varphi}(x,y)$ converges to $K(x-y)H_{f,\varphi}(x,y)$ outside the set $\{(x,y)\in \mathbb{R}^d\times\mathbb{R}^d:x=y\}$.  Insofar as the measure $\mu$ is nice, this set has zero $\mu\times\mu$ measure, and also $\tfrac{\chi_S(x)\chi_S(y)}{|x-y|^{s-1}}\in L^1(\mu\times\mu)$.  Consequently, the dominated convergence theorem ensures that
$$\langle T_{\mu, \delta}(f),\varphi \rangle_{\mu} = \iint_{\mathbb{R}^d\times\mathbb{R}^d} K_{\delta}(x-y)H_{f,\varphi}(x,y) d\mu(x)d\mu(y)
$$
converges as $\delta\rightarrow 0$ to
$$ \iint_{\mathbb{R}^d\times\mathbb{R}^d}K(x-y)H_{f,\varphi}(x,y)d\mu(x)d\mu(y)=\langle T(f\mu),\varphi \rangle_{\mu}.
$$
Furthermore, since $|\langle T_{\mu, \delta}(f),\varphi \rangle_{\mu} |\leq C_0\|f\|_{L^2(\mu)}\|\varphi\|_{L^2(\mu)}$ for every $\delta>0$,
$$|\langle T(f\mu),\varphi \rangle_{\mu} |\leq C_0\|f\|_{L^2(\mu)}\|\varphi\|_{L^2(\mu)}.
$$
Consequently, by the Riesz-Fisher theorem, these exists a unique bounded linear operator $T_{\mu}:L^2(\mu)\rightarrow L^2(\mu)$ such that
$$\langle T_{\mu}(f), \varphi \rangle_{\mu} = \langle T(f\mu), \varphi \rangle_{\mu} \text{ whenever }f,\varphi\in \Lip_0(\mathbb{R}^d).
$$

\subsection{Uniform rectifiability and the local convexity condition}

Now fix $s\in \mathbb{Z}$.  In this section we recall some of the language and results from David and Semmes \cite{DS}.

We say that an AD-regular measure $\mu$ is \textit{uniformly rectifiable} if there exists $M>0$ such that for every cube $Q\in \mathcal{D}$, there is a Lipschitz mapping $F_Q:\mathbb{R}^s\rightarrow \mathbb{R}^d$ with $\|F_Q\|\leq M\ell(Q)$ and $\mu(F_Q(\R^s)\cap Q)\geq \tfrac{1}{2}\mu(Q)$.

We shall now recall one of the criterion for uniform rectifiability given in \cite{DS}.  Among the several equivalent conditions for uniform rectifiability given in \cite{DS}, the most convenient condition to work with when taking the weak limit of a sequence of measures appears to be the \textit{Local Weak Convexity} (LCV) condition:

Fix $\delta>0$.   For an AD-regular measure $\mu$, we say that a dyadic cube $Q\in \mathcal{D}$ is $\delta$-non-LCV if
\begin{equation}\begin{split}&\text{there exist points }x,y\in \overline{3Q}\cap\supp(\mu)\\& \text{ such that } B(\tfrac{x+y}{2},\delta\ell(Q))\cap \supp(\mu)=\varnothing.
\end{split}\end{equation}

According to Corollary 2.10 of Chapter 1 in \cite{DS} the is uniformly rectifiablity of $\supp(\mu)$ is equivalent to the fact that for each $\delta>0$, the family of $\delta$-non LCV dyadic cubes is a \textit{Carleson family}, i.e., for each $\delta>0$, there exists $C_{\delta}>0$ such that for every $P\in \mathcal{D}$
$$\sum_{\substack{Q\in \mathcal{D}: Q\subset P,\\Q \text{ is }\delta-\text{non LCV}}}\ell(Q)^s\leq C_{\delta}\ell(P)^s.
$$

\subsection{Stabilization of Dyadic Lattices}\label{stabdyadic}

We say that a sequence of dyadic lattices $\mathcal{D}_k$ stabilizes tn a dyadic lattice $\mathcal{D}'$ if every $Q'\in \mathcal{D}'$ lies in $\mathcal{D}_k$ for sufficiently large $k$.

\begin{lem} Suppose $\mathcal{D}_k$ is a sequence of dyadic lattices with $Q_0=(0,1)^d\in \mathcal{D}_k$ for all $k$.  Then there exists a subsequence of the lattices that stabilizes to some lattice $\mathcal{D}'$.
\end{lem}

The lemma is proved via a standard diagonal argument:  For $n\geq0$, there are only $2^{nd}$ distinct ways to arrange a dyadic lattice so that $Q_0$ is a child of a cube of sidelength $2^n$.

\section{Main Results}

Having introduced the required notation and concepts, we can list our main results.  Firstly, the integer dimensional rigidity result.

\begin{prop}\label{intcase}  Let $s\in \mathbb{Z}$, $s\in (0,d)$.  Suppose that $T$ is an $s$-dimensional CZO, and that the only $s$-AD regular reflectionless measures associated to $T$ 
are of the form $c\mathcal{H}^1|L$  for a constant $c>0$ and an $s$-plane $L$.

If $\mu$ is an $s$-AD regular measure for which $T$ is bounded in $L^2(\mu)$, then for every $\delta>0$, the family of dyadic $\delta$-non LCV  cubes is a Carleson family.
\end{prop}


The second result is the non-integer rigidity result.  Formally, there is no need in this case to impose the requirement that $s$ be non-integer, but in order for the hypothesis on $T$ to be satisfied, $s$ cannot be an integer.  Recall the Wolff potential (\ref{powerwolff}) from the introduction.

\begin{prop}\label{thm1}  Let $s\in (0,d)$.  Suppose that $T$ is an $s$-dimensional CZO, and that the only nice reflectionless measure for $T$ is the zero measure.

There exists $p\in (0,\infty)$, depending on $d$, $s$, and $\alpha$, such that any $\Lambda$-nice measure $\mu$ with associated CZO transform $T$ bounded in $L^2(\mu)$ satisfies
\begin{equation}\label{powerwolffest}\int_{Q}\mathbb{W}_{p}(\chi_{Q}\mu)d\mu\leq C\mu(Q) \text{ for any cube }Q\subset \R^d,
\end{equation}
for a constant $C>0$ depending on $s$, $d$, $\alpha$, $\Lambda$, and the $\|T_{\mu}\|_{L^2(\mu)\to L^2(\mu)}$.
\end{prop}

As we discussed in the introduction, we are able to answer Question \ref{reflconj} affirmatively if $s\in (d-1,d)$.

\begin{prop}\label{onlytrivmeas}  If $s\in (d-1,d)$, then the only nice reflectionless measure for the $s$-dimensional Riesz transform is the zero measure.
%
\end{prop}

\section{An overview of the proof of Proposition \ref{thm1}}\label{outline}

 We shall prove Proposition \ref{thm1} first.  The proof of Proposition \ref{intcase} will be significantly simpler.   It will be convenient to prove the analogue of Proposition \ref{thm1} with a dyadic Wolff potential.

\begin{prop}\label{dyprop}    Suppose that the only nice reflectionless measure for a CZO $T$ is the zero measure.   If $\mu$ is a finite nice measure for which the CZO $T$ is bounded in $L^2(\mu)$, then there exist $p\in (0,\infty)$ depending on $s,d,$ and $\alpha$, and a constant $C>0$, depending on $s$, $d$, $\alpha$ and $\|T\|_{\mu}$, such that
$$\sum_{Q\in \mathcal{D}}D_{\mu}(3Q)^p \mu(3Q)\leq C\mu(\mathbb{R}^d).
$$
\end{prop}

To get Proposition \ref{thm1} from Proposition \ref{dyprop}, merely note that if $\mu$ is a measure with associated CZO transform $T$ bounded in $L^2(\mu)$, then for any cube $Q$, $\nu =\chi_{Q}\mu$ is a finite measure for which the associated CZO transform $T$ is bounded in $L^2(\nu)$.  Proposition \ref{dyprop} then yields that
$$\int_{\R^d}\mathbb{W}_p(\nu)d\nu\leq C\sum_{Q'\in \mathcal{D}}D_{\nu}(3Q')^p\nu(3Q')\leq C\nu(\R^d),
 $$
as required.

The proof of this proposition proceeds through studying the Lipschitz oscillation coefficient at a dyadic cube.  We first outline this scheme.

Let $\mu$ be a (non-negative) measure.  For $A>100\sqrt{d}$, and a cube $Q\in \mathcal{D}$, define the set of functions
$$\Phi_A^{\mu}(Q) = \Bigl\{\psi\in \Lip_0(B(x_Q, A\ell(Q))) : \|\psi\|_{\Lip}\leq\frac{1}{\ell(Q)}, \int_{\mathbb{R}^d}\psi d\mu=0\Bigl\}.$$
The system  $\Phi_A^{\mu}(Q)$ $(Q\in \mathcal{D})$ forms a Riesz system, that is, there exists a constant $C=C(A,s,d)>0$, such that for any $f\in L^2(\mu)$, and every choice of $\psi_Q\in \Psi_Q$ ($Q\in \mathcal{D}$),
$$\sum_{Q\in \mathcal{D}}\frac{|\langle f, \psi_Q\rangle_{\mu}|^2}{\mu(B(x_{Q}, 3A\ell(Q)))} \leq C\|f\|_{L^2(\mu)}^2.$$
See Appendix \ref{riesz} for a proof of this fact.  Consequently, if in addition $\mu$ is a finite $\Lambda$-nice measure for which $T$ is bounded in $L^2(\mu)$, then
\begin{equation}\begin{split}\label{rieszTbd}\sum_{Q\in \mathcal{D}}&\frac{|\langle T(\psi_Q\mu), 1 \rangle_{\mu}|^2}{\mu(B(x_{Q}, 3A\ell(Q)))} =  \sum_{Q\in \mathcal{D}}\frac{|\langle T_{\mu}(1), \psi_Q\rangle_{\mu}|^2}{\mu(B(x_{Q}, 3A\ell(Q)))} \\&\leq C(A)\|T_{\mu}(1)\|_{L^2(\mu)}^2\leq C(A)\|T_{\mu}\|^2_{L^2(\mu)\rightarrow L^2(\mu)}\mu(\mathbb{R}^d).
\end{split}\end{equation}

We now introduce the \emph{Lipschitz Oscillation Coefficient}.  Define
\begin{equation}\label{liposcdef}\Theta_{\mu}^{A}(Q) =\sup_{f\in \Phi_A^{\mu}(Q)}|\langle T(f\mu),1 \rangle_{\mu}|.
\end{equation}

From (\ref{rieszTbd}) we see that
\begin{equation}\label{rieszTbd2}
\sum_{Q\in \mathcal{D}}\frac{\Theta_{\mu}^A(Q)^2}{\mu(B(x_{Q}, 3A\ell(Q)))} \leq C(A)\|T_{\mu}\|^2_{L^2(\mu)\rightarrow L^2(\mu)}\mu(\mathbb{R}^d).
\end{equation}

\begin{lem}\label{coeflem}Suppose that $\mu$ is a finite $\Lambda$-nice measure for which the CZO $T$ is bounded in $L^2(\mu)$.  Suppose that $\mathcal{F}\subset \mathcal{D}$, $\Delta>0$ and $A>1$ are such that for every $Q\in \mathcal{F}$,
\begin{equation}\label{coeflowbd}\Theta_{\mu}^A(Q)\geq \Delta D_{\mu}(3Q)\mu(3Q).
\end{equation}
Then
$$\sum_{Q\in \mathcal{F}} D_{\mu}(3Q)^2 \frac{\mu(3Q)^2}{\mu(B(x_Q, 3A\ell(Q)))}\leq C(A)\Bigl(\frac{\|T_{\mu}\|_{L^2(\mu)\rightarrow L^2(\mu)}}{\Delta}\Bigl)^2\mu(\mathbb{R}^d).
$$
\end{lem}

The lemma is an immediate consequence of the inequality (\ref{rieszTbd2}).  Consequently, \textit{if we were able to show that} there exist positive constants $\Delta$ and $A$, so that (\ref{coeflowbd}) holds \textit{for every cube} $Q\in \mathcal{D}$, then we would arrive at
$$\sum_{Q\in \mathcal{D}}D_{\mu}(3Q)^2 \frac{\mu(3Q)^2}{\mu(B(x_{Q}, 3A\ell(Q)))}\leq C\mu(\mathbb{R}^d),
$$
from which the solution of the Mateu-Prat-Verdera conjecture would follow for the CZO $T$.  However,\textit{ the inequality (\ref{coeflowbd}) can easily fail for many cubes}, but the main point behind the proof of Proposition \ref{dyprop}  is that under the hypothesis of the non-existence of non-trivial $\Lambda$-nice reflectionless measures, (\ref{coeflowbd}) holds for a class of cubes near which $\mu$ is sufficiently regular.

\begin{defn}Let $\mu$ be a locally finite measure, and let $\eps>0$.  A cube $Q\in \mathcal{D}$ is called $\eps$-regular for $\mu$ if
\begin{equation}\label{regcube}D_{\mu}(3Q') \leq 2^{\eps d(Q,Q')}D_{\mu}(3Q), \text{ for every }Q'\in \mathcal{D}.
\end{equation}
\end{defn}

In Section \ref{nonstandard} we shall show that if any locally finite measure $\mu$ has an $\eps$-regular cube $Q \in \mathcal{D}$, and $\eps$ is small enough (in terms of $d,s$ and $\alpha$), then $\mu$ is diffuse and so one can define the Lipschitz oscillation coefficient at $Q$.  We shall then prove the following alternative:


\begin{prop}\label{wavelet}  One of the following two statements holds:  Either

(i)  There exists a non-trivial nice reflectionless measure for the CZO $T$,

or

(ii) There exist $\eps>0$,  $A>0$, and $\Delta>0$, such that whenever $\mu$ is a locally finite measure and $Q\in \mathcal{D}$ is an $\eps$-regular cube, one has $$\Theta_{\mu}^{A}(Q)\geq \Delta D_{\mu}(3Q)\mu(3Q).$$
\end{prop}

Under the hypothesis of Proposition \ref{dyprop}, we are forced into part (ii) of this alternative.  Let us now fix $\eps$ and $A$ as in the statement of part (ii).  Returning to the case of a finite nice measure $\mu$ for which $T$ is bounded in $L^2(\mu)$, Lemma \ref{coeflem} implies that
$$\sum_{Q\in \mathcal{D}: Q\text{ is }\eps\text{-regular}}D_{\mu}(3Q)^2 \frac{\mu(3Q)^2}{\mu(B(x_{Q}, 3A\ell(Q)))}\leq C\mu(\mathbb{R}^d),
$$
where the constant $C>0$ may depend on $\eps$, $A$, $\|T_{\mu}\|_{L^2(\mu)\rightarrow L^2(\mu)}$, and $\Delta$.  However, a regular cube is doubling by its defining property, and so
$$\sum_{Q\in \mathcal{D}: Q\text{ is }\eps\text{-regular}}D_{\mu}(3Q)^2 \mu(3Q)\leq C\mu(\mathbb{R}^d)
$$
as well.

We therefore arrive at the following question:   When does the sum over regular $\eps$-cubes bound the entire sum?  In Section \ref{graph}, we shall prove the following lemma.
\begin{lem}\label{pdomination} Suppose that $\mu$ is a finite nice measure.  For each $\eps>0$, there exists $p=p(\eps,d)\geq 2$ such that
$$\sum_{Q\in \mathcal{D}}D_{\mu}(3Q)^p\mu(3Q)\leq C\sum_{\substack{Q\in \mathcal{D}:\\ Q \text{ is }\eps-\text{regular}}}D_{\mu}(3Q)^p\mu(3Q),
$$
where $C=C(\eps,d, s)>0$.
\end{lem}

From this lemma, and the inequality $D_{\mu}(3Q)^p\leq (C\Lambda)^{p-2}D_{\mu}(3Q)^2$, we have that
$$\sum_{Q\in \mathcal{D}}D_{\mu}(3Q)^p\mu(3Q)\leq C\sum_{\substack{Q\in \mathcal{D}:\\ Q \text{ is }\eps-\text{regular}}}D_{\mu}(3Q)^2\mu(3Q)\leq C\mu(\R^d),
$$
and Proposition \ref{dyprop} follows.



\section{Measures with slightly non-standard growth}\label{nonstandard}





Fix $\beta$ satisfying $\beta<\min(\alpha, s)$ (recall here that $\alpha\leq 1$). We say that a measure $\mu$ is $\Lambda$-\textit{reasonable} if
$$\mu(B(x,r)) \leq \Lambda r^s\bigl(R\max(1,\tfrac{1}{r})\bigl)^{\beta},
$$
whenever $B(x,r)\subset B(0,R)$, with $R>1$. \\

It is immediate from the definition that if $\mu_k$ is a sequence of $\Lambda$-reasonable measures, then there is a subsequence which converges weakly to a $\Lambda$-reasonable measure $\mu$.  (Of course, here we are using the weak compactness of the space of locally finite measures  over the (separable) space of compactly supported continuous functions.)

Throughout this section, all constants may depend on $\Lambda$ and $\beta$ without explicit mention.  The next lemma consists of two straightforward estimates using the definition of a reasonable measure.

\begin{lem}\label{reasint}  Let $R>1$, and suppose that $\mu$ is $\Lambda$-reasonable.  Then there is a constant $C>0$ such that if $r\in (0, 2R)$,
$$\int_{|x-y|<r}\frac{1}{|x-y|^{s-1}}d\mu(y) \leq CR^{\beta}\min(r^{1-\beta},r) \text{ for }x\in B(0,R),
$$
and,
$$\int_{\mathbb{R}^d\backslash B(0,R)}\frac{1}{|x|^{s+\alpha}}d\mu(x) \leq \frac{C}{R^{\alpha-\beta}}.
$$
\end{lem}

\begin{proof}  For the first estimate, note that
$$\int_{|x-y|<r}\frac{1}{|x-y|^{s-1}}d\mu(y)\leq C\int_0^{r}\frac{\mu(B(x,t))}{t^{s-1}}\frac{dt}{t}.
$$
Note that $B(x,t)\subset B(0,3R)$ for $x\in B(0,R)$ and $t\leq 2R$.  Thus, if $t\in (0,1]$, then $\mu(B(x,t))\leq 3^{\beta}R^{\beta}\Lambda t^{s-\beta}$.  Consequently, if $r\in (0,1]$
$\int_0^r\tfrac{\mu(B(x,t))}{t^{s-1}}\tfrac{dt}{t}\leq CR^{\beta}\int_0^rt^{1-\beta}\tfrac{dt}{t}\leq CR^{\beta}r^{1-\beta}$.  On the other hand, for $t\in (1,2R)$, $\mu(B(x,t))\leq 3^{\beta}R^{\beta}\Lambda t^{s}$, and so $\int_1^{r}\tfrac{\mu(B(x,t))}{t^{s-1}}\tfrac{dt}{t}\leq CR^{\beta}r$.  The desired bound follows.

The tail estimate is just as simple.  For $r>1$, we have $\mu(B(0,r)) \leq \Lambda r^{s+\beta}$. Substituting this inequality into the integral $\int_R^{\infty}\frac{\mu(B(0,r))}{r^{s+\alpha}}\frac{dr}{r}$ yields the required estimate.
\end{proof}

An immediate consequence of the first estimate in this lemma is that a $\Lambda$-reasonable measure $\mu$ is diffuse, that is, the function $(x,y)\rightarrow \frac{1}{|x-y|^{s-1}}$ is locally integrable with respect to $\mu\times\mu$.  Furthermore, the second estimate ensures that $\mu$ has restricted growth at infinity in the sense that $\int_{|x|\geq 1}\tfrac{1}{|x|^{s+\alpha}}d\mu(x)<\infty$.  Consequently,  the bilinear form $\langle T(f\mu), 1\rangle_{\mu}$ is well defined for any $f\in \Lip_0(\mathbb{R}^d)$ with $\int fd\mu=0$ (recall Section \ref{reflmeas}).\\


Now, recall from Section 8 of Part I that a sequence of measures $\mu_k$ is called \textit{uniformly diffuse} if, for each $R>0$ and $\eps>0$, there exists $r>0$ such that for all $k$,
$$\iint\limits_{\substack{B(0,R)\times B(0,R)\\|x-y|<r}}\frac{d\mu_k(x)d\mu_k(y)}{|x-y|^{s-1}}
\leq \eps,
$$
and $\mu_k$ is said to have \textit{uniformly restricted growth (at infinity)} if, for each $\eps>0$, there exists an $R\in(0,\infty)$ such that for all $k$,
$$\int_{\mathbb{R}^d\backslash \overline{B(0, R)}}\frac{1}{|x|^{s+\alpha}}d\mu_k(x)
\leq \eps.
$$

Notice that since the constant in Lemma \ref{reasint} depends only on $s$, $\Lambda$, and $\beta$, a sequence $\mu_k$ of $\Lambda$-reasonable measures is uniformly diffuse with uniformly restricted growth at infinity.  Consequently, Lemma 8.2 of Part I is applicable to such a sequence of measures.  Let us now state this convergence lemma for the special case under consideration.

Consider two sets of functions $$\Phi^{\mu}_R =\Bigl\{\psi\in \Lip_0(B(0, R)): \|\psi\|_{\Lip} < 1\;\text{ and } \int_{B(0,R)}\psi \,d\mu=0\Bigl\},$$
and $$\Phi^{\mu} =\Bigl\{\psi\in \Lip_0(\mathbb{R}^d): \|\psi\|_{\Lip} < 1\;\text{ and } \int_{\mathbb{R}^d}\psi\, d\mu=0\Bigl\} = \bigcup_{R>0}\Phi_R^{\mu}.$$

\begin{lem}\label{smalldiff}  Suppose that $\mu_k$ are $\Lambda $-reasonable measures that converge weakly to a measure $\mu$ (and so $\mu$ is $\Lambda$-reasonable as well).  Let $\gamma_k$ and $\widetilde{R}_k$ be sequences of non-negative numbers satisfying $\gamma_k\rightarrow 0 $, and $\widetilde{R}_k\rightarrow +\infty$, as $k\rightarrow\infty$.

If $|\langle T(\psi\mu_k),1\rangle_{\mu_k}|\leq \gamma_k$ for all $\psi\in \Phi_{\widetilde{R}_k}^{\mu_k}$, then $\langle T(\psi \mu),1 \rangle_{\mu} =0$ for all $\psi\in \Phi^{\mu}$.
\end{lem}

Let us also record a useful corollary of this result.

\begin{cor} \label{reflweaklim}  The weak limit of a sequence of $\Lambda$-reasonable reflectionless measures is (provided it exists) a $\Lambda$-reasonable reflectionless measure.\end{cor}

\subsection{Lipschitz oscillation coefficients and reflectionless measures}


In this section we prove Proposition \ref{wavelet}.  We shall assume that statement (i) of the proposition fails to hold, which is to say that the only nice reflectionless measure is the trivial measure.


The proof that statement (ii) holds will be obtained via a compactness argument.  First we fix $\eps$, and let $A$ tend to infinity and $\Delta$ tend to zero.  Then we let $\eps$ tend to zero.

\begin{lem}\label{regularreasonable}  There exists $\Lambda>0$ so that if $\eps$ is small enough (smaller than some $\eps_0$ depending on $\beta$), then any measure $\mu$ satisfying
$$D_{\mu}(3Q)\leq 2^{\eps d(Q,Q_0)} \text{ for every } Q\in \mathcal{D},
$$
for some dyadic lattice $\mathcal{D}$ containing $Q_0$, is $\Lambda$-reasonable.
\end{lem}

\begin{proof}Fix a ball $B(x,r)\subset B(0,R)$ with $R>1$.  Then $B(x,r)$ is contained in the union of at most $3^d$ dyadic cubes of side length between $r$ and $2r$.  We shall estimate $d(Q,Q_0)$ for one of these dyadic cubes $Q$.  Note that $Q$ is contained in the ball $B(0,10\sqrt{d}R)$, and so has graph distance at most $\log_2(R/r)+ C$ from the dyadic ancestor of $Q_0$ of sidelength between $R$ and $2R$.  But then $d(Q,Q_0)\leq 2\log_2 R+\log_2(1/r)+C$.  It follows that
 $$\mu(B(x,r))\leq Cr^{s-\eps}R^{2\eps},
 $$so we only need to choose $\eps_0\leq \tfrac{\beta}{2}$.\end{proof}

\begin{lem}\label{epsregalt}  Let $\eps\in (0,\eps_0)$.  One of the following two statements holds:

(i)  There exist $A= A(\eps)$ and $\Delta=\Delta(\eps)>0$, such that every non-trivial locally finite measure $\mu$ with an $\eps$-regular cube $Q\in \mathcal{D}$ satisfies
  $$\Theta_{\mu}^{A}(Q)\geq \Delta D_{\mu}(3Q)\mu(3Q).$$

(ii)  There exists a reflectionless measure satisfying $\mu(\overline{3Q_0})\geq 1$ and
\begin{equation}\label{egrow}D_{\mu}(3Q)\leq 2^{\eps d(Q,Q_0)} \text{ for every }Q\in \mathcal{D'}\end{equation}
where $\mathcal{D}'$ is some dyadic lattice containing $Q_0$.

\end{lem}

\begin{proof}
Suppose that (i) fails to hold.  For each $k>100\sqrt{d}$, there is a non-trivial measure $\widetilde\mu_k$, and an $\eps$-regular cube $Q_k$, with $$|\langle T(\psi{\widetilde{\mu}_k}),1\rangle_{\widetilde{\mu}_k}|\leq \frac{1}{k}  D_{\widetilde{\mu}_k}(3Q_k)\widetilde{\mu}_k(3Q_k)\text{ for all }\psi\in \Phi_{k}^{\widetilde{\mu}_k}(Q_k).$$

Now define the measure $\mu_k$ by $$\mu_k(\,\cdot\,) = \displaystyle \tfrac{1}{\widetilde{\mu}_k(3Q_k)}\widetilde{\mu}_k(x_{Q_k}+\ell(Q_k)\, \cdot \, ).$$  Then $\mu_k(\overline{3Q_0})\geq 1$.  Furthermore, $\mu_k$ satisfies the inequality (\ref{egrow}) in the shifted lattice $\mathcal{D}_k = \tfrac{1}{\ell(Q_k)}[\mathcal{D}-x_{Q_k}]$, and
$$|\langle T(\psi\mu_k),1\rangle_{\mu_k}|\leq \frac{1}{k}  \text{ for all }\psi\in \Phi_{k}^{\mu_k}.
$$

By choosing a suitable subsequence, we may assume that $\mu_k$ converge weakly to a measure $\mu$ with $\mu(\overline{3Q_0})\geq 1$.  Passing to a further subsequence if necessary, we may assume that the lattices $\mathcal{D}_k$ stabilize to some lattice $\mathcal{D}'$, see Section \ref{stabdyadic}.   Since the dyadic cubes are open, the lower semicontinuity of the weak limit ensures that $\mu$ satisfies (\ref{egrow}) in the lattice $\mathcal{D}'$.  The measures $\mu_k$ are $\Lambda$-reasonable (Lemma \ref{regularreasonable}), and so applying Lemma \ref{smalldiff} with $\gamma_k = \tfrac{1}{k}$, and $\widetilde{R}_k =k$ yields the reflectionless measure promised in statement (ii).
\end{proof}



Our second lemma rules out the possibility that the second alternative of  Lemma \ref{epsregalt} holds for every $\eps>0$, and so proves Proposition \ref{wavelet}.

\begin{lem}\label{epsnonexist} Suppose that the only nice reflectionless measure is the zero measure.  Then there exists $\eps\in (0, \eps_0)$ such that there is no reflectionless measure $\mu$ satisfying $\mu(\overline{3Q_0})\geq 1$ and \begin{equation}\label{lambdagrow}D_{\mu}(3Q)\leq 2^{\eps d(Q, Q_0)}\end{equation} for every $Q\in \mathcal{D}$ where $\mathcal{D}$ is a dyadic lattice with $Q_0\in \mathcal{D}$.
\end{lem}


\begin{proof} Suppose that for each $\eps\in (0,\eps_0)$, there exists a reflectionless measure $\mu_{\eps}$ with $\mu_{\eps}(\overline{3Q_0})\geq 1$ satisfying (\ref{lambdagrow}) for every cube $Q$ in some dyadic lattice $\mathcal{D}_{\eps}$ containing $Q_0$.  We may assume that $\mu_{\eps}$ converge weakly to a measure $\mu$ as $\eps$ tends to zero along a suitably chosen sequence, and also that the lattices $\mathcal{D}_{\eps}$ stabilize to some lattice $\mathcal{D}'$.  Since the measures $\mu_{\eps}$ are $\Lambda$-reasonable, an application of Corollary \ref{reflweaklim} ensures that the limit measure $\mu$ is reflectionless.  However, $\mu(3Q)\leq \ell(Q)^s$ for any cube $Q\in \mathcal{D}'$.  Thus $\mu$ is a nice reflectionless measure.  But $\mu(\overline{3Q_0})\geq 1$.   This contradiction proves the lemma.\end{proof}

\section{Senior vertices on a graph}\label{graph}

To complete the proof of Proposition \ref{dyprop} it remains to provide a proof of Lemma \ref{pdomination}, which is a very elementary piece of graph theory.\smallskip

Suppose that $\Gamma$ is a graph with vertex degree bounded by $D$.  Suppose that $\nu$ is a bounded non-negative function on $\Gamma$.

Let $M>0$.  We call a vertex $x\in \Gamma$ \textit subordinate to $y\in \Gamma$ if $\nu(x)< 2^{-Md(x,y)}\nu(y)$.  Here $d(x,y)$ denotes the graph distance (i.e., the length of a shortest path from $x$ to $y$ in $\Gamma$).  A vertex $x\in \Gamma$ is \textit{senior} if it is not subordinate to any vertex in the graph.

For each $x\in \Gamma$, consider $\max\{\nu(y)2^{-Md(x,y)} : y\in \Gamma\}$.  That the maximum is attained is an immediate consequence of the boundedness of $\nu$ and the vertex degree.  Suppose that the maximum is attained at $x^*$. (We shall view $x^{*}$ as a vertex determined by the vertex $x$.)  Then we claim that $x^*$ is senior.  Otherwise, there exists some $z\in \Gamma$ such that $\nu(x^*)< 2^{-Md(x^*,z)}\nu(z)$.  But then by the triangle inequality $\nu(x^*)2^{-Md(x,x^*)}<\nu(z)2^{-Md(x,z)}$, which is a contradiction.  Clearly $\nu(x)\leq 2^{-Md(x^*,x)}\nu(x^*)$.

If $2^{-M}D<1$, then for any fixed senior vertex $z\in \Gamma$,
$$\sum_{x\in \Gamma : x^{*}=z}\nu(x) \leq \sum_{x\in \Gamma} 2^{-Md(x,z)}\nu(z) \leq \nu(z) \sum_{k\geq 0}2^{-Mk}\#\{x\in \Gamma: d(x,z)=k\},$$
which is at most $\nu(z)\sum_{k\geq 0} 2^{-Mk}D^k\leq C\nu(z)$.  Combining these observations, we arrive at

$$\sum_{x\in \Gamma}\nu(x) \leq C \sum_{x \text{ is senior}}\nu(x),
$$


\subsection{The proof of Lemma \ref{pdomination}}

Fix $\eps>0$ and a finite nice measure $\mu$.

First choose $M$ with $(2^d+2d+1)2^{-M}<1$.  (The number $2^d+2d+1$ is an upper bound for the vertex degree of the graph $\Gamma(\mathcal{D})$.)  For $p\geq 2$, set $$\nu(Q) = D_{\mu}(3Q)^p\mu(3Q).$$  Notice that $\nu$ is a bounded function, since $\mu$ is a finite nice measure.  A vertex $Q\in \mathcal{D}$ is senior if $$D_{\mu}(3Q')^p\mu(3Q')\leq 2^{Md(Q',Q)}D_{\mu}(3Q)\mu(3Q)\text{ for every }Q'\in \mathcal{D}.$$

The general considerations of Section \ref{graph} guarantee that
\begin{equation}\label{uniform}\sum_{Q\in \mathcal{D}}\nu(Q) \leq C\sum_{Q\in \mathcal{D}: \,Q\text{ is senior}}\nu(Q).
\end{equation}

On the other hand, for a senior cube $Q\in \mathcal{D}$,
$$D_{\mu}(3Q')^{p+1} \leq 2^{Md(Q', Q)}D_{\mu}(3Q)^{p+1}\frac{\ell(Q)^s}{\ell(Q')^s}\leq 2^{(M+s)d(Q,Q')}D_{\mu}(3Q)^{p+1},
$$
for any $Q'\in \mathcal{D}$.  Thus, a senior cube $Q$ is an $\tfrac{M+s}{p+1}$-regular cube for $\mu$.

But now, provided that $\tfrac{M+s}{p+1}<\eps$, we have that
$$\sum_{Q\in \mathcal{D}}\nu(Q) \leq C\sum_{Q\in \mathcal{D}: \,Q\text{ is }\eps-regular}\nu(Q),
$$
and so Lemma \ref{pdomination} is proved.

\section{The proof of Proposition \ref{intcase}}

The proof of Proposition \ref{intcase} is quite similar to that of Proposition \ref{thm1} except that the proof is significantly more qualitative, and the measures under consideration will have more regularity.

Fix $s\in \mathbb{Z}$, $s\in (0,d)$.  Let us suppose that $T$ is a CZO such that the only $s$-AD regular reflectionless measures associated to $T$ 
are of the form $c\mathcal{H}^s|L$  for a constant $c>0$ and an $s$-plane $L$.  Recall the definition of the Lipschitz oscillation coefficient $\Theta_{\mu}^A(Q)$ from Section \ref{outline}.  We shall prove the following lemma.

\begin{lem}\label{LCVdelta}  For each $\delta>0$, there exists $\Delta>0$ and $A>0$, such that if $\mu$ is an $\Lambda$-AD regular measure, and $Q\in \mathcal{D}$ is $\delta$-non LCV for $\mu$,  then
$$\Theta_{\mu}^A(Q) \geq \Delta\ell(Q)^s.
$$
\end{lem}

Taking this lemma for granted for the time being, we shall conclude the proof of Proposition \ref{intcase}.  Let us recall that we want to show that if $\mu$  is a $\Lambda$-AD regular measure with $T_{\mu}$ bounded in $L^2(\mu)$, then  for each $\delta>0$, there exists $C_{\delta}>0$ such that for every $P\in \mathcal{D}$
\begin{equation}\label{deltcarl}\sum_{\substack{Q\in \mathcal{D}: Q\subset P,\\Q \text{ is }\delta-\text{non-LCV (for }\mu)}}\ell(Q)^s\leq C_{\delta}\ell(P)^s.
\end{equation}
To see this, fix some $\Lambda$-AD regular measure with $T$ bounded in $L^2(\mu)$.  Also fix $\delta>0$, and a dyadic cube $P\in \mathcal{D}$.  Consider a cube $Q\in \mathcal{D}$, $Q\subset P$ that is $\delta$-non LCV.  By Lemma \ref{LCVdelta}, there exists $\psi_Q\in \Phi^{\mu}_A(Q)$ with
$$|\langle T(\psi_Q\mu),1\rangle_{\mu}|\geq \frac{\Delta}{2}\ell(Q)^s.
$$
Now, set $A'>100\sqrt{d}A$.   Choose $\varphi_{A'}\in \Lip_0(B(x_P, 2A'\ell(P)))$ satisfying $0\leq\varphi_{A'}\leq 1$ on $\mathbb{R}^d$ and $\varphi\equiv 1$ on $B(x_{P}, A'\ell(P))$.  Then, since $\psi_Q$ has $\mu$-mean zero, $|\langle T(\psi_Q\mu),1-\varphi_{A'}\rangle_{\mu}|$ is dominated by
$$\int_{\mathbb{R}^d\backslash B(x_P, A'\ell(P))}\int_{B(x_Q, A\ell(Q))}|K(y-x)-K(-x)||\psi_Q(y)|d\mu(y)d\mu(x).
$$
But, with $x\not\in  B(x_P, A'\ell(P))$, and $y\in B(x_Q, A\ell(Q))$, $|K(y-x)-K(-x)|\leq\tfrac{CA^{\alpha}\ell(Q)^{\alpha}}{|x|^{s+\alpha}}$.  Also, $\|\psi_Q\|_{\infty}\leq 2A$.  Thus,
$$|\langle T(\psi_Q\mu),1-\varphi_{A'}\rangle_{\mu}|\leq CA\int_{\mathbb{R}^d\backslash B(x_P, A'\ell(P))}\frac{A^{\alpha+s}\ell(Q)^{s+\alpha}}{|x|^{s+\alpha}}d\mu(x),$$
but this is dominated by $\tfrac{CA^{1+s+\alpha}\ell(Q)^s}{A'^{\alpha}}$ (as $\ell(P)\geq \ell(Q)$).  Fix $A'$ (chosen in terms of $\delta$, $s$, $d$, and $\Lambda$), so that $$\frac{CA^{1+s+\alpha}}{A'^{\alpha}}\leq\frac{\Delta}{4}.$$
Our conclusion is that, for each $\delta$-non LCV cube $Q\subset P$, there exists $\psi_Q\in \Phi^{\mu}_A(Q)$ such that
$$|\langle T(\psi_Q\mu),\varphi_{A'}\rangle_{\mu}|\geq \frac{\Delta}{4}\ell(Q)^s.
$$
Now, recall that the system  $\Phi_A^{\mu}(Q)$ $(Q\in \mathcal{D})$ forms a Riesz system, so there exists a constant $C=C(A,s,d)>0$, such that for every choice of $\psi_Q\in \Phi_A^{\mu}(Q)$ ($Q\in \mathcal{D}$)
$$\sum_{Q\in \mathcal{D}: \, Q\subset P}\frac{|T(\varphi_{A'}\mu), \psi_{Q}\rangle_{\mu}|^2}{\mu(B(x_{Q}, 3A\ell(Q)))} \leq C\|T_{\mu}(\varphi_{A'})\|_{L^2(\mu)}^2.
$$
Restricting the sum to those $\delta$-non LCV cubes $Q$ contained in $P$, we deduce that
\begin{equation}\begin{split}\nonumber \sum_{\substack{Q\in \mathcal{D}: Q\subset P,\\Q \text{ is }\delta-\text{non-LCV}}}\frac{\ell(Q)^{2s}}{{\mu(B(x_{Q}, 3A\ell(Q)))} } \leq C(\delta)\|T_{\mu}(\varphi_{A'})\|_{L^2(\mu)}^2\\
\leq C(\delta)\|T_{\mu}\|^2_{L^2(\mu)\rightarrow L^2(\mu)}\|\varphi_{A'}\|_{L^2(\mu)}^2\leq C(\delta)\|T_{\mu}\|^2_{L^2(\mu)\rightarrow L^2(\mu)}\ell(P)^s.
\end{split}\end{equation}
However, since $\mu(B(x_{Q}, 3A\ell(Q)))\leq C(A)\ell(Q)^s$, we derive the required inequality (\ref{deltcarl}).\medskip

We now return proving Lemma \ref{LCVdelta}.  Let's begin with a few simple facts about the weak convergence of AD-regular measures.
\begin{itemize}
\item  Suppose that $\mu_k$ is a sequence of $\Lambda$-AD regular measures.  Then there is a subsequence of the measures $\mu_k$ that converges weakly to a $\Lambda$-AD regular measure $\mu$.
\item Fix $\mu_k$ a sequence of $\Lambda$-AD regular measures that converges weakly to a measure $\mu$ (and so $\mu$ is $\Lambda$-AD regular).  Suppose that $x_k\in \supp(\mu_k)$ and $x_k$ converges to some $x\in \mathbb{R}^d$.  Then $x\in \supp(\mu)$.
\item Fix a dyadic cube $Q\in \mathcal{D}$.  Let $\mu_k$ be a sequence of $\Lambda$-AD regular measures that converges weakly to a measure $\mu$ (and so $\mu$ is $\Lambda$-AD regular).  If $Q$ is $\delta$-non LCV for each $\mu_k$, then $Q$ is $\delta$-non LCV for $\mu$.
\end{itemize}

The first two facts are essentially immediate and very well known.  We shall prove the third item.  By definition, there are points $x_k$ and $y_k$ in $\overline{3Q}\cap \supp(\mu_k)$ such that $z_k = \tfrac{x_k+y_k}{2}$ satisfies $B(z_k, \delta\ell(Q))\cap\supp(\mu_k)=\varnothing$.   By passing to a subsequence, we may assume that $x_k$ converge to some $x\in \overline{3Q}\cap \supp(\mu)$, and $y_k$ converge to some $y\in \overline{3Q}\cap \supp(\mu)$.  But then $z_k$ converges to $z=\tfrac{x+y}{2}$.  Now, choose an increasing sequence $f_{\ell}\in \Lip_0(B(z,\delta\ell(Q)))$ that converges pointwise to $\chi_{B(z,\delta\ell(Q))}$.  For each $\ell$, $\supp(f_{\ell})\subset B(z_k, \delta\ell(Q))$ for all sufficiently large $k$, and so $\int_{\mathbb{R}^d}f_{\ell}\,d\mu = \lim_{k\rightarrow\infty}\int_{\R^d}f_{\ell}\,d\mu_k=0$.  But then the monotone convergence theorem ensures that $\mu(B(z,\delta))=0$.

We now suppose that the statement of the lemma is false.  Then for some $\delta>0$, and every $k\in \mathbb{N}$, there exists a $\Lambda$-AD regular measure $\widetilde{\mu}_k$, and a dyadic cube $Q_k$ that is $\delta$-non LCV for $\widetilde{\mu}_k$, such that
$$|\langle T(\psi \widetilde{\mu}_k),1\rangle_{\widetilde{\mu}_k}|\leq\frac{1}{k}\ell(Q_k)^s
$$
for every $\psi\in \Phi_k^{\widetilde{\mu}_k}(Q_k)$.

For each $k$, consider the measure $\mu_k = \tfrac{\widetilde{\mu}_k(x_{Q_k}+\ell(Q_k)\,\cdot\,)}{\ell(Q_k)^s}$.  Then $\mu_k$ is $\Lambda$-AD regular, the unit cube $Q_0$ is $\delta$-non LCV for $\mu_k$, and
$$|\langle T(\psi \mu_k),1\rangle_{\mu_k}|\leq\frac{1}{k} \text{ for all }\psi\in \Phi_k^{\mu_k}(Q_0),
$$
(and in particular this inequality holds for $\psi\in \Psi^{\mu_k}_k$).

By passing to a subsequence if necessary, we may assume that there is a $\Lambda$-AD regular measure $\mu$ such that $Q_0$ is $\delta$-non LCV for $\mu$, and $\mu_k$ converges to $\mu$ weakly.

On the other hand, Lemma \ref{smalldiff} is applicable (a sequence of $\Lambda$-AD-regular measures are certainly $\Lambda$-reasonable) with $\gamma_k = \tfrac{1}{k}$, and $R_k=k$.  Applying the lemma, we conclude that $\mu$ is reflectionless.  By hypothesis, $\mu$ therefore takes the form $\mu=c\mathcal{H}^s|L$ for some $s$-plane $L$ and $c>0$.  But this measure cannot have a $\delta$-non LCV cube.  This contradiction proves the lemma, and with it the proposition.

\section{An Extremal Reflectionless Measure} \label{extremalsection} We now prove the existence of an extremal $\Lambda$-nice reflectionless measure for smooth non-degenerate CZOs.  This extremal measure will form a key tool in the argument asserting Proposition \ref{onlytrivmeas}.  We shall use the results of Section 6 from Part I here.

 We shall also use the notation of Section 3.7 of Part I, with the function $m_d$-almost everywhere defined function $\T1_{\mu}(1)$.  One should think of $\T1_{\mu}(1)(x)$  as the difference $\int_{\R^d}K(y-x)d\mu(y) - \langle T(\varphi_0\mu), 1\rangle_{\mu}$ where $\varphi_0\in \Lip_0(\R^d)$ satisfies $\int_{\R^d}\varphi_0d\mu=1$.   Notice that the local part of the first term in the difference, say $\int_{\R^d}K(y-x)\psi_R(y)d\mu(y)$ where $\psi_R\in \Lip_0(B(0,2R))$ satisfies $\psi_R\equiv 1$ on $B(0,R)$, lies in $L^1_{\text{loc}}(m_d)$ as
 $$\int_{K}\int_{B(0,R)}\frac{1}{|x-y|^s}d\mu(y)dm_d(x)\leq Cm_d(K)^{\tfrac{d-s}{d}}\mu(B(0, R)),
 $$
 for any compact set $K\subset \R^d$.  The term $\langle T(\varphi_0\mu), \psi_R\rangle_{\mu}$ also makes sense as a bilinear form. The remaining contribution to the difference $\int_{\R^d}K(y-x)d\mu(y) - \langle T(\varphi_0\mu), 1\rangle_{\mu}$ can be written as
\begin{equation}\label{T1taildouble}\iint_{\R^d\times \R^d}[K(y-x)-K(y-z)](1-\psi_R(y))\varphi_0(z)d\mu(y)d\mu(z),
 \end{equation}
 and provided that $R$ is chosen so large as to ensure that $x\in B(0, \tfrac{R}{2})$ and $\supp(\varphi_0)\subset B(0, \tfrac{R}{2})$ this double integral converges absolutely due to the restricted growth at infinity.  The precise definition of $\T1_{\mu}(1)$ as an $m_d$-almost everywhere defined function can therefore be taken as
\begin{equation}\begin{split}\nonumber\T1_{\mu}(1)(x) & = \int_{\R^d}K(y-x)\psi_R(y)d\mu(y)+ \langle T(\varphi_0\mu), \psi_R\rangle_{\mu}\\&+\iint_{\R^d\times \R^d}[K(y-x)-K(y-z)](1-\psi_R(y))\varphi_0(z)d\mu(y)d\mu(z),
 \end{split}\end{equation}
 where $R$ is chosen sufficiently large.  The value $\T1_{\mu}(1)(x)$ is of course independent of the choice of $R$ as long as the double integral (\ref{T1taildouble}) converges absolutely.

 Outside of applying the results from Part I verbatim, the only fact that the reader needs to know in this section about $\T1_{\mu}(1)$ is that for $m_d$-almost every $x, x'\in \mathbb{R}^d$, $\T1_{\mu}(1)(x) - \T1_{\mu}(1)(x') = \int_{\mathbb{R}^d} [K(y-x)-K(y-x')]d\mu(y)$, which can be readily checked from the definition given above.

The Cotlar Lemma, Crollary 7.2 of Part I, states that if $\mu$ is a $\Lambda$-nice reflectionless measure, then $\|\T1_{\mu}(1)\|_{L^{\infty}(m_d)}\leq C$ where $C>0$ depends only on $d$, $s$, and $\Lambda$.

In this subsection, we shall assume that $\Omega \in C^{\infty}(\mathbb{S}^{d-1})$ satisfies $m\bigl(\frac{\xi}{|\xi|}\bigl)\neq 0$ for any $\xi\in \mathbb{R}^d$, where $m$ and $\Omega$ are related by \begin{equation}\label{fourier}
\mathcal{F}\Bigl(\frac{\Omega\bigl(\tfrac{\cdot}{|\cdot|}\bigl)}{|\cdot|^s}\Bigl)(\xi) = \frac{m\bigl(\tfrac{\xi}{|\xi|}\bigl)}{|\xi|^{d-s}}, \text{ for any } \xi\neq 0,
\end{equation}
here $\mathcal{F}$ is the Fourier transform operator.  This assumption guarantees that the Wiener lemma (Lemma 7.3 of Part I) holds.

\begin{prop}\label{genprop}   If there is a non-trivial $\Lambda$-nice reflectionless measure, then there exists a $\Lambda$-nice reflectionless measure $\mu^{\star}$ such that $\dist(0,\supp(\mu^{\star}))= 1$ and
$$|\T1_{\mu^{\star}}(1)(0)| = \|\T1_{\mu^{\star}}(1)\|_{L^{\infty}(m_d)}.
$$
\end{prop}

We now set up an extremal problem whose solution will provide the measure $\mu^{\star}$ whose existence is claimed in the statement of Proposition \ref{genprop}.  

Define $\mathcal{F}$ to be the set of non-trivial $\Lambda$-nice reflectionless measures $\mu$.  We suppose that $\mathcal{F}\neq \varnothing$.

Set $\mathcal{Q} = \sup\{|\T1_{\mu}(1)(0)|: \mu\in \mathcal{F}\text{ with } \dist(0, \supp(\mu))=1\}$.

\begin{cla}  $\mathcal{Q}>0$.
\end{cla}

\begin{proof} Pick a measure $\mu\in \mathcal{F}$.  The Wiener Lemma yields that if $|\T1_{\mu}(1)|=0$ $m_d$-almost everywhere in $\mathbb{R}^d$, then $\mu=0$. If $\mu\in \mathcal{F}$, then $ |\T1_{\mu}(1)|=0$ $m_d$-almost everywhere on $\supp(\mu)$ (Corollary 6.6 from Part I), and so there must be a point $z \not\in \supp(\mu)$ with $|\T1_{\mu}(1)(z)|>0$.  Set $p=\dist(z, \supp(\mu))$.  Consider the measure $\tilde\mu(\cdot) = \tfrac{\mu(p \cdot+ z)}{p^s}$.  Then $\tilde\mu \in \mathcal{F}$, $\dist(0, \supp(\mu))=1$, and $|\T1_{\tilde\mu}(1)(0)|=|\T1_{\mu}(1)(z)| >0$.
\end{proof}


\begin{cla}  $\mathcal{Q}<+\infty$.
\end{cla}

\begin{proof}  This follows immediately from the Cotlar Lemma (Corollary 7.2 from Part I).
\end{proof}


\begin{cla}\label{extremecla}  There exists $\mu^{\star}\in \mathcal{F}$ with $\dist(0, \supp(\mu))= 1$, such that $|\T1_{\mu^{\star}}(1)(0)|=\mathcal{Q}$.
\end{cla}

\begin{proof}
For each $j\in \mathbb{N}$, choose $\mu_j\in \mathcal{F}$ with $\dist(0, \supp(\mu_j))=1$, satisfying $|\T1_{\mu_j}(1)(0)|\geq \mathcal{Q}(1-2^{-j-1})\geq \tfrac{\mathcal{Q}}{2}$.
Then, by Corollary 6.5 in Part I, there exists $M'=M'(\mathcal{Q})$ such that $\mu_j(B(0, M'))\geq c(\mathcal{Q})$ for each $j$.  We may pass to a subsequence that converges to a $\Lambda$-nice reflectionless measure $\mu$ (Corollary \ref{reflweaklim} of this paper).  From standard weak semi-continuity properties of the weak limit, we have that $\dist(0, \supp(\mu))\geq 1$, and $\mu(\overline{B(0,M')})\geq c(\mathcal{Q})$.  Since the sequence of measures $\mu_j$ are uniformly diffuse (see Section \ref{nonstandard}, where it is shown that even a sequence of $\Lambda$-reasonable measures is uniformly diffuse) the convergence result Lemma 8.1 in Part I is applicable, and yields that $|\T1_{\mu}(1)(0)|=\mathcal{Q}$.  Fix $p=\dist(0,\supp(\mu))$.  Setting $\mu^{\star}(\,\cdot\,) = \tfrac{\mu(p\,\cdot\,)}{p^s}$ yields the claim.
\end{proof}

\begin{proof}[The proof of Proposition \ref{genprop}]    Consider the measure $\mu^{\star}$ constructed in Claim \ref{extremecla}, and suppose that $|\T1_{\mu^{\star}}(1)(0)|<\|\T1_{\mu^{\star}}(1)\|_{L^{\infty}(m_d)}$.  As a result of Corollary 6.6 from Part I, there exists $x\not\in \supp(\mu^{\star})$ with $|\T1_{\mu^{\star}}(1)(x)|>|\T1_{\mu^{\star}}(1)(0)|$.  But now set $p=\dist(x, \supp(\mu^{\star}))$.  Consider $\tilde\mu(\,\cdot\,) = \tfrac{\mu^{\star}(p\,\cdot\,+x)}{p^s}$.  Then $\tilde\mu\in \mathcal{F}$, and $\mathcal{Q}<|\T1_{\mu^{\star}}(1)(x)| = |\T1_{\tilde\mu}(1)(0)|$.  This is absurd.
\end{proof}

\section{The Riesz transform}

In this section, we consider the simplest, and most interesting $s$-dimensional CZO, the $s$-Riesz transform.  This is the choice of kernel $K(x)=\tfrac{x}{|x|^{s+1}}$ for $x\in \mathbb{R}^d$ (so the Riesz transform is $\mathbb{R}^d$-valued).  We will write $R_{\mu}$ instead of $T_{\mu}$, $\overline{R}_{\mu}(1)$ instead of $\T1_{\mu}(1)$, and so on.

Note that $\Omega$ is smooth, and $m\bigl(\tfrac{\xi}{|\xi|}\bigl) = c\tfrac{\xi}{|\xi|}$ for a non-zero complex number $c$, where $m$ is given by (\ref{fourier}).  Thus the Wiener Lemma, and hence Proposition \ref{genprop}, are both applicable for the $s$-Riesz transform when $s\in (0,d)$.


\subsection{The proof of Proposition \ref{onlytrivmeas}}

We shall need a lemma which accounts for the restriction to $s\in (d-1,d)$.  It is nothing more than the integral representation formula of the fractional Laplacian (from which the strong maximum principle trivially follows), but we couldn't find the statement precisely in the form we need it, so a proof is included in an appendix.

\begin{lem}\label{sharmon} Suppose that $s\in (d-1,d)$, and $\mu$ is a $\Lambda $-nice measure, with $0\not\in \supp(\mu)$.  Then
\begin{equation}\begin{split}\nonumber P.V. \int_{\mathbb{R}^d} &\frac{\overline{R}_{\mu}(1)(0)  - \overline{R}_{\mu}(1)(x)}{|x|^{2d+1-s}} dm_d(x)\\
&=\lim_{\delta\rightarrow 0}\int_{\mathbb{R}^d\backslash B(0,\delta)}\frac{\overline{R}_{\mu}(1)(0) - \overline{R}_{\mu}(1)(x)}{|x|^{2d+1-s}} dm_d(x)=0.
\end{split}\end{equation}
\end{lem}


\begin{proof}[Proof of Proposition \ref{onlytrivmeas}]  Suppose that there is a nontrivial reflectionless measure.  Consider the measure $\mu^{\star}$ provided by Proposition \ref{genprop}.  Since $|\overline{R}_{\mu^{\star}}(1)(0)| = \|\overline{R}_{\mu^{\star}}(1)\|_{L^{\infty}(m_d)}$, Lemma \ref{sharmon} implies that $\overline{R}_{\mu^{\star}}(1)$ is constant $m_d$-almost everywhere.  But then the Wiener Lemma yields that $\mu\equiv 0$.  This is a contradiction.
\end{proof}

\subsection{Weak Porosity}\label{porous}

Having proved that non-trivial $\Lambda$-nice reflectionless measures for the $s$-Riesz transform fail to exist if $s\in (d-1,d)$, we move onto a studying them for $s\leq d-1$.  We shall here prove that the support of a reflectionless for the Riesz transform is nowhere dense.  We actually prove a slightly stronger version of this statement.



\begin{prop}\label{Rieszporous}  Suppose that $\mu$ is a $\Lambda$-nice reflectionless measure for the $s$-Riesz transform, with $s\in (0,d-1]$.  For each $\eps>0$ there is a constant $\lambda=\lambda(\eps)>0$ such that if $\mu(B(x,r))>\eps r^s$, then there is a ball $B'\subset B(x,3r)$ of radius $\lambda r$ that does not intersect $\supp(\mu)$.
\end{prop}

Taking into account the general Porosity result in Lemma 6.7 of Part I, Proposition \ref{Rieszporous} will follow immediately from the following result.


\begin{lem}  Let $s\in (0,d-1]$.  There is a constant $c>0$, such that if $\mu(B(x,r))\geq \eps r^s$, then  $\int_{B(x,3r)}|\overline{R}_{\mu}(1)| dm_d>c\eps m_d(B(x,3r))$.
\end{lem}

\begin{proof}  We may assume that $x=0$ and $r=1$.  Let $\psi_{\tfrac{1}{2}}$ be a non-negative bump function supported in $B(0,\tfrac{1}{2})$, with $\int_{\mathbb{R}^d} \psi_{\tfrac{1}{2}}dm_d=1$.  Then $(\psi_{1/2}*\mu)(B(0,2))\geq c\eps$.  There is a positive constant $b=b(s)$ such that
$$\text{div}(\psi_{\tfrac{1}{2}}*\overline{R}_{\mu}(1))(x) = \begin{cases} b(\psi_{\tfrac{1}{2}}*\mu)(x) \text{ if }s=d-1\\ b\int_{\mathbb{R}^d} \tfrac{(\psi_{1/2}*\mu)(y)}{|x-y|^{s+1}}dm_d(y) \text{ if }s<d-1.\end{cases}
$$
Indeed, for fixed $x'\in \mathbb{R}^d$,
$$\psi_{\tfrac{1}{2}}*\overline{R}_{\mu}(1)(x)- \psi_{\tfrac{1}{2}}*\overline{R}_{\mu}(1)(x') = \int_{\mathbb{R}^d}[K(y-x)-K(y-x')]d(\psi_{\tfrac{1}{2}}*\mu)(y),
$$
from which the formulas follow from differentiating the kernel.

On the other hand, if $\varphi\in C^{\infty}_0(\mathbb{R}^d)$, then
$$\int_{\mathbb{R}^d}[\psi_{1/2}*\overline{R}_{\mu}(1)]\cdot \nabla \varphi dm_d = - \int_{\mathbb{R}^d} \text{div}(\psi_{1/2}*\overline{R}_{\mu}(1)) \varphi dm_d.$$
Choose $\varphi$ to be nonnegative, with bounded gradient, and satisfying $\varphi \equiv 1$ on $B(0,2)$, $\supp(\varphi)\subset B(0,3)$.  Then
\begin{equation}\begin{split}\nonumber C\int_{B(0,3)}|\overline{R}_{\mu}(1)| dm_d &\geq \Bigl|\int_{\mathbb{R}^d}[\psi_{1/2}*\overline{R}_{\mu}(1)]\cdot \nabla \varphi dm_d\Bigl|\\&\geq \int_{B(0,2)}\text{div}(\psi_{1/2}*\overline{R}_{\mu}(1))  dm_d\geq c\eps,
\end{split}\end{equation}
as required.
\end{proof}

\section{Behaviour at Infinity} \label{behaviourinfinity} The growth of a reflectionless measure at infinity is something we do not yet understand particularly well.  Studying tangent measures at infinity formed an important part of Preiss's proof of the rectifiability of a measure $\mu$ for which the limit $\lim_{r\rightarrow 0}D_{\mu}(B(x,r))$ exists for $\mu$-almost every $x\in \R^d$ \cite{Pre}, and there is a hope that studying the behaviour of reflectionless measures at infinity could help shed some light on Question \ref{reflconj}.

In this section, we make some elementary remarks about the behaviour of reflectionless measures at infinity in order to introduce a couple of simple ideas.

\begin{lem}\label{nomeanzero} Suppose that $\mu$ is a reflectionless measure for a CZO $T$ satisfying the following (uniform diffuseness at infinity) condition:  For every $\eps>0$ there exists $\delta>0$ such that $$\frac{1}{\mu(B(0,R))}\iint\limits_{\substack{x,y\in B(0,R):\\ |x-y|< \delta R}}\frac{1}{R|x-y|^{s-1}}d\mu(y)d\mu(x)<\eps$$
for all sufficiently large $R>0$.

If
\begin{equation}\label{finitesint}\int_{\R^d}\frac{1}{(1+|x|)^s}d\mu(x)<\infty,
\end{equation}
then
$$\langle T(\varphi\mu), 1\rangle_{\mu}=0 \text{ for all }\varphi\in \Lip_0(\R^d) $$
(and not only those test functions with $\mu$-mean zero).
\end{lem}

Here one makes sense of $\langle T(\varphi\mu), 1\rangle_{\mu}$ by first introducing some $\psi\in \Lip_0(\R^d)$ that is identically one on the support of $\varphi$.  The condition (\ref{finitesint}) yields that $x\mapsto K(x-y)(1-\psi(x))\in L^1(\mu)$ is $y\in \supp(\varphi)$.   Therefore we may set
$$\langle T(\varphi\mu), 1\rangle_{\mu}=\langle T(\varphi\mu), \psi\rangle_{\mu} +\int_{\R^d}\varphi(y)\int_{\R^d}K(x-y)(1-\psi(x))d\mu(x)d\mu(y).
$$

Let us remark that if $s\leq 1$, then (\ref{finitesint}) already implies that $\mu$ is uniformly diffuse at infinity.  For any $s$, the uniform diffuseness condition is satisfied if $\mu$ has finite energy in the sense that
\begin{equation}\label{finitenergy}\int_{\mathbb{R}^d\times\mathbb{R}^d}\frac{1}{|x-y|^{s-1}}d\mu(x)d\mu(y)<\infty
\end{equation}
Any $\Lambda$-reasonable measure is also uniformly diffuse at infinity.

\begin{proof}  Fix $\varphi\in \Lip_0(\R^d)$.  We shall also fix a non-negative function $\psi\in \Lip_0(B(0,\tfrac{3}{2}))$ that equals $1$ everywhere on $B(0,1)$.  For $R>0$, we shall set $\psi_R(\,\cdot\,) = \psi\bigl(\tfrac{\cdot}{R}\bigl)$.

For any $R>0$ large enough to ensure that $\int_{\R^d}\psi_R d\mu>0$ the reflectionless property of $\mu$ guarantees that
$$\Bigl\langle T\Bigl(\Bigl[\varphi - \frac{\int_{\R^d}\varphi d\mu}{\int_{\R^d}\psi_Rd\mu}\psi_R\Bigl]\mu\Bigl),1\Big\rangle_{\mu}=0.
$$
Consequently, to prove the result it shall suffice to find a sequence of radii $R_j\rightarrow \infty$ such that
$$\lim_{j\rightarrow \infty}\frac{1}{\int_{\R^d}\psi_{R_j}d\mu}\langle T(\psi_{R_j}\mu),1\rangle_{\mu}=0.
$$
Since $\sup_{R>1}\frac{\mu(B(0,R))}{R^s}<\infty$, we can find an infinite sequence of radii $R_j= 3^{\ell_j}$, $\ell_j\in \mathbb{N}$, that are \emph{doubling} in the sense that $$\mu(B(0, 3R_j))\leq 9^s \mu(B(0, R_j)).$$  Indeed, for each non-doubling radius $R$, $D_{\mu}(B(0, 3R))>3^sD_{\mu}(B(0, R))$, and so there can be no infinite sequence of consecutive non-doubling radii.

Notice that $\mu(B(0, 3R_j))\leq 9^s\int_{\R^d}\psi_{R_j}d\mu\leq 81^s\mu(B(0, R_j))$.  In addition, as $\sum_j D_{\mu}(B(0, 3R_j))<\infty$, there is a sequence $\delta_j\to 0$ such that $D_{\mu}(B(0, 3R_j))\delta_j^{-(s-1)}\to 0$ as $j\to\infty$.

Let us now group together the simple estimates we shall require.  Fix a doubling radius $R_j$, then
\begin{equation}\begin{split}\nonumber|\langle T(\psi_{R_j}\mu)&, 1-\psi_{2R_j}\rangle_{\mu}|\leq \int_{|x|\geq 2R_j}\int_{|y|<\tfrac{3}{2}R_j}\frac{1}{|x-y|^s}d\mu(y)d\mu(x)\\&\leq C\mu(B(0, \tfrac{3}{2}R_j))\int_{|x|>R_j}\frac{1}{|x|^s}d\mu(x)\\&\leq C\Bigl[\int_{|x|>R_j}\frac{1}{|x|^s}d\mu(x)\Bigl]\mu(B(0,R_j)).
\end{split}\end{equation}
On the other hand, we write $\langle T(\psi_{R_j}\mu), \psi_{2R_j}\rangle_{\mu}$ as
$$\iint\limits_{B(0,3R_j)\times B(0, 3R_j)} K(x-y)(\psi_{R_j}(y)\psi_{2R_j}(x) - \psi_{R_j}(x)\psi_{2R_j}(y))d\mu(x)d\mu(y),
$$
which is bounded in absolute value by a constant multiple of
$$\iint\limits_{B(0,3R_j)\times B(0, 3R_j)}\frac{1}{R_j|x-y|^{s-1}}d\mu(x)d\mu(y).
$$
Now notice that
$$\iint\limits_{\substack{B(0,3R_j)\times B(0, 3R_j)\\|x-y|> \delta_j \cdot R_j}}\frac{1}{R_j|x-y|^{s-1}}d\mu(x)d\mu(y)\leq C\frac{D_{\mu}(B(0, 3R_j))}{\delta_j^{s-1}}\mu(B(0,R_j)).
$$
Bringing these estimates together we see that, for each $j$,
\begin{equation}\begin{split}\nonumber\frac{1}{\int_{\R^d}\psi_{R_j}d\mu}|&\langle T(\psi_{R_j}\mu), 1\rangle_{\mu}|\leq  C\int_{|x|\geq R_j}\frac{1}{|x|^s}d\mu(x)+ C\frac{D_{\mu}(B(0, 3R_j))}{\delta_j^{s-1}}\\&+C\frac{1}{\mu(B(0, 3R_j))}\iint\limits_{\substack{B(0,3R_j)\times B(0, 3R_j)\\|x-y|\leq \delta_j \cdot R_j}}\frac{1}{R_j|x-y|^{s-1}}d\mu(x)d\mu(y).
\end{split}\end{equation}
We see that $\int_{|x|\geq R_j}\frac{1}{|x|^s}d\mu(x)$ tends to zero as $j\to\infty$ due to (\ref{finitesint}).  On the other hand, the diffuseness at infinity ensures that$$\frac{1}{\mu(B(0, 3R_j))}\iint\limits_{\substack{B(0,3R_j)\times B(0, 3R_j)\\|x-y|\leq \delta_j \cdot R_j}}\frac{1}{R_j|x-y|^{s-1}}d\mu(x)d\mu(y)\rightarrow 0 \text{ as }j\to\infty.
$$
But $D_{\mu}(B(0, 3R_j))\delta_j^{-(s-1)}\to 0$ as $j\rightarrow \infty$ by construction, and so we conclude that $$\lim_{j\rightarrow \infty}\frac{1}{\int_{\R^d}\psi_{R_j}d\mu}\langle T(\psi_{R_j}\mu),1\rangle_{\mu}=0,
$$
as required.
\end{proof}

We now move onto using this lemma to prove the following proposition.

\begin{prop}\label{norieszfinitenergy}
Let $s\in (0,d)$.  The only reflectionless measure $\mu$ for the $s$-Riesz transform satisfying (\ref{finitenergy}) is the zero measure.
\end{prop}

\begin{proof}Let us suppose that $\mu$ is a finite reflectionless measure satisfying (\ref{finitenergy}).  Fix $\varphi\in \Lip(\R^d)$, and choose a sequence  $\varphi_n\in \Lip_0(\mathbb{R}^d)$ that satisfies $\sup_n \|\varphi_n\|_{\Lip}<\infty$ and $\varphi_n(x)\rightarrow \varphi(x)$.  From Lemma \ref{nomeanzero} we have that $\langle R(\varphi_n\mu),1\rangle_{\mu}=0$ for each $n$.

Now notice that
$$\bigl|[[\varphi-\varphi_n](x)-[\varphi-\varphi_n](y)]K(x-y)\bigl|\rightarrow 0 \text{ whenever }x\neq y,
$$
and the set $\{(x,y)\in \mathbb{R}^d\times\mathbb{R}^d: x=y\}$ is a set of $\mu\times\mu$ measure zero.  In addition, there is a constant $C>0$ such that for $x\neq y$
$$\sup_n\bigl|[[\varphi-\varphi_n](x)-[\varphi-\varphi_n](y)]K(x-y)\bigl| \leq \frac{C}{|x-y|^{s-1}}.
$$
The function $(x,y)\mapsto\frac{1}{|x-y|^{s-1}}$ lies in $L^1(\mu\times\mu)$ due to the condition (\ref{finitenergy}), and so the dominated convergence theorem now yields that $$\iint\limits_{\R^d\times\R^d}K(x-y)[\varphi(x)-\varphi(y)]d\mu(x)d\mu(y)=0.$$


Now fix a co-ordinate $j\in \{1,\dots, d\}$, and consider the function $\varphi(x) = x_j$.  Then 
$$\int_{\mathbb{R}^d\times\mathbb{R}^d} \frac{(x_j-y_j)(x-y)}{|x-y|^{s+1}}d\mu(x)d\mu(y)=0.
$$
In particular, taking the $j$-th co-ordinate of this vector yields $$\int_{\mathbb{R}^d\times\mathbb{R}^d} \frac{(x_j-y_j)^2}{|x-y|^{s+1}}d\mu(x)d\mu(y)=0.
$$
But then summation over $j$ now yields
$$\int_{\mathbb{R}^d\times\mathbb{R}^d} \frac{1}{|x-y|^{s-1}}d\mu(x)d\mu(y)=0,
$$
and $\mu$ must be the zero measure.
\end{proof}

\begin{rem}The harmonic measure in $\R^2$ of the line segment $[-1,1]$ with pole at infinity is the measure $\mu$ that lies on the line $\{x_2=0\}$ with density  $d\mu(x_1,x_2) = \frac{1}{\pi}\frac{1}{\sqrt{1-x_1^2}}\chi_{[-1,1]}(x_1)dm_1(x_1)$.  We recall from \cite{MPV} that this measure has the property that the principal value of the one dimensional Riesz transform of $\mu$ is equal to zero on $\supp(\mu)$.  The above proposition in particular yields that $\mu$ is \textit{not} reflectionless in our sense\footnote{In particular, one could make a valid complaint about our use of the terminology reflectionless.}.
\end{rem}

\begin{ques}  Suppose that $\mu$ is a $\Lambda$-nice reflectionless measure for the $s$-Riesz transform with $s\in (1,d-1)$.  If $\mu\neq0$, does there exist $\eps>0$ such that
$$\mu(B(0,R))\geq R^{\tfrac{s}{2}+\eps}\text{ for all sufficiently large } R>0?
$$
\end{ques}

We would be especially interested if one could answer this question with $\eps=\tfrac{s}{2}$.

\appendix

\section{The sufficiency of the Mateu-Prat-Verdera condition}\label{squarewolffcond}

The purpose of this appendix is to provide the proof of the following well-known result.  Fix $s\in (0,d)$, and let $T$ be an $s$-dimensional CZO.

\begin{thm}  Suppose that there is a constant $C_0>0$ so that for every cube  $Q\subset \R^d$,
\begin{equation}\label{Wolfftesting}\int_{Q}\mathbb{W}_2(\chi_{Q}\mu)d\mu(x)\leq C_0\mu(Q).
\end{equation}
Then $T_{\mu}$ is bounded in $L^2(\mu)$, with norm bounded by $C\cdot C_0$, where $C$ depends on $d$, $s$ and $\alpha$.
\end{thm}

Although this theorem is indeed well known, it is difficult to locate a proof, so we shall provide one here.  The proof is very similar to that of Theorem 4.6 of \cite{ENV}, but we are working under a slightly weaker assumption on the Wolff potential, and so prefer to avoid any integration by parts arguments.  We shall rely upon the following lemma.

\begin{lem}  There is a constant $C>0$ depending on $d$, $s$ and $\alpha$, such that for any finite measure $\nu$ and $\eps>0$,
$$\int_{\R^d}\Bigl|\int_{|x-y|>\eps}K(x-y)d\nu(y)\Bigl|^2d\nu(x)\leq C\int_{\R^d}\mathbb{W}_2(\nu)d\nu.
$$
\end{lem}

To prove the theorem from the lemma, fix a cube $Q$ and consider the measure $\nu = \chi_{Q}\mu$.  From the lemma, and (\ref{Wolfftesting}), we find a constant $C>0$, depending only of $C_0$, $d$, $s$, and $\alpha$, such that for any $\eps>0$,
$$\int_{Q}\Bigl|\int_{Q: |y-x|>\eps}K(x-y)d\mu(y)\Bigl|^2d\mu(x)\leq C\mu(Q).
$$
But then the non-homogeneous $T(1)$-theorem \cite{NTV} yields that the CZO $T$ associated to $\mu$ is bounded in $L^2(\mu)$.

\begin{proof}  Let us expand the left hand side:
$$\iiint\limits_{\substack{x,y,z\in \R^d: \\|x-y|>\eps,\, |x-z|>\eps}}K(x-y)\cdot\overline{K(x-z)}d\nu(x)d\nu(y)d\nu(z).
$$
It is enough to estimate the absolute value of the integral with the domain of integration restricted to
$$U=\bigl\{(x,y,z)\in \R^{3d}:|x-y|\geq |x-z|>\eps\bigl\}.$$
First consider the set
$$U_1=\bigl\{(x,y,z)\in U: |y-z|<|x-z|\bigl\}.
$$
Notice that for $(x,y,z)\in U_1$, $|x-z|>\tfrac{1}{2}|x-y|$.  Thus $$|K(x-y)\cdot\overline{K(x-z)}|\leq \frac{C}{|x-y|^{2s}}.$$
But then
\begin{equation}\begin{split}\nonumber\Bigl|\iiint\limits_{\substack{(x,y,z)\in U_1}}&K(x-y)\cdot\overline{K(x-z)}d\nu(x)d\nu(y)d\nu(z)\Bigl|\\
&\leq \iiint\limits_{(x,y,z)\in \R^{3d}\,:\,|x-y|\geq |x-z|}\frac{1}{|x-y|^{2s}}d\nu(x)d\nu(y)d\nu(z).
\end{split}\end{equation}
However, the right hand side is of course dominated by a constant multiple of
$$\int_{\R^d}\int_0^{\infty}\frac{1}{r^{2s}}\iint\limits_{(y,z)\in \R^{2d}:\,|x-y|<r, |x-z|<r}d\nu(y)d\nu(z)\frac{dr}{r}d\nu(x),
$$
but this integral equals $\int_{\R^d}\mathbb{W}_2(\nu)d\nu$.

It remains to estimate
$$\Bigl|\iiint\limits_{U_2}K(x-y)\cdot\overline{K(x-z)}d\nu(x)d\nu(y)d\nu(z)\Bigl|,
$$
where
\begin{equation}\begin{split}\nonumber U_2 &= \bigl\{(x,y,z)\in U: |y-z|\geq |x-z|\bigl\}\\
&=\bigl\{(x,y,z)\in \R^{3d}: |y-z|\geq |x-z|,\; |x-y|\geq |x-z|\bigl\}.
\end{split}\end{equation}
It is at this point where we shall appeal to the facts that $K$ is antisymmetric and H\"{o}lder continuous away from the diagonal.  Notice that the set $U_2$ is symmetric under permuting $x$ and $z$.  Thus, we may estimate
$$\frac{1}{2}\iiint\limits_{U_2}|K(x-y)\cdot\overline{K(x-z)}+ K(z-y)\cdot\overline{K(z-x)}|d\nu(x)d\nu(y)d\nu(z)
$$
However, for $(x,y,z)\in U_2$, the integrand, $|K(x-y)\cdot\overline{K(x-z)}+ K(z-y)\cdot\overline{K(z-x)}|$, is bounded by
\begin{equation}\begin{split}\nonumber|\overline{K(x-z)}||K(x-y)-K(z-y)|&\leq \frac{1}{|x-z|^s}\frac{C|x-z|^{\alpha}}{|x-y|^{s+\alpha}}\\&\leq \frac{C}{|x-z|^{s-\alpha}|x-y|^{s+\alpha}}.\end{split}\end{equation}
Our goal is now to establish the pointwise estimate
\begin{equation}\nonumber
\iint\limits_{(y,z)\in \R^{2d}: |x-y|\geq |x-z|}\frac{1}{|x-z|^{s-\alpha}|x-y|^{s+\alpha}}d\nu(z)d\nu(y)\leq C\mathbb{W}_2(\nu)(x)
\end{equation}
for every $x\in \R^d$.
Appealing to the distribution formula, we first bound the left hand side of the desired inequality by a constant multiple of
$$\int_0^{\infty}\frac{\nu(B(x,r))}{r^{s+\alpha}}\int_{|x-z|<r}\frac{1}{|x-z|^{s-\alpha}}d\nu(z)\frac{dr}{r}.
$$
But,
$$\int_{|x-z|<r}\frac{1}{|x-z|^{s-\alpha}}d\nu(z)=(s-\alpha)\int_0^r\frac{\nu(B(x,t))}{t^{s-\alpha}}\frac{dt}{t}+\frac{\nu(B(x,r))}{r^{s-\alpha}},
$$
and since
$$\int_0^{\infty}\frac{\nu(B(x,r))}{r^{s+\alpha}} \frac{\nu(B(x,r))}{r^{s-\alpha}}\frac{dr}{r} = \mathbb{W}_2(\nu)(x),
$$
it suffices to estimate
$$\int_0^{\infty}\frac{\nu(B(x,r))}{r^{s+\alpha}}\int_{0}^r\frac{\nu(B(x,t))}{t^{s-\alpha}}\frac{dt}{t}\frac{dr}{r}.$$
We first use Cauchy's inequality:
$$\frac{\nu(B(x,r))}{r^{s+\alpha}}\int_{0}^r\frac{\nu(B(x,t))}{t^{s-\alpha}}\frac{dt}{t}\leq \Bigl(\frac{\nu(B(x,r))}{r^{s}}\Bigl)^2+\Bigl(\frac{1}{r^{\alpha}}\int_0^r\frac{\nu(B(x,t))}{t^s}\frac{dt}{t^{1-\alpha}}\Bigl)^2,
$$
to reduce matters to estimating
$$\int_0^{\infty}\Bigl(\frac{1}{r^{\alpha}}\int_0^r\frac{\nu(B(x,t))}{t^s}\frac{dt}{t^{1-\alpha}}\Bigl)^2\frac{dr}{r}.
$$
But $\int_0^r\frac{dt}{t^{1-\alpha}}=\frac{1}{\alpha}r^{\alpha}$, so the Cauchy-Schwarz inequality yields that
$$\Bigl(\frac{1}{r^{\alpha}}\int_0^r\frac{\nu(B(x,t))}{t^s}\frac{dt}{t^{1-\alpha}}\Bigl)^2\leq C\frac{1}{r^{\alpha}}\int_0^r\Bigl(\frac{\nu(B(x,t))}{t^s}\Bigl)^2\frac{dt}{t^{1-\alpha}},
$$
from which we deduce that
\begin{equation}\begin{split}\nonumber\int_0^{\infty}\Bigl(\frac{1}{r^{\alpha}}&\int_0^r\frac{\nu(B(x,t))}{t^s}\frac{dt}{t^{1-\alpha}}\Bigl)^2\frac{dr}{r}\\&\leq C \int_0^{\infty}\Bigl(\frac{\nu(B(x,t))}{t^s}\Bigl)^2\Bigl[\int_t^{\infty}\frac{1}{r^{\alpha}}\frac{dr}{r}\Bigl]\frac{dt}{t^{1-\alpha}}.
\end{split}\end{equation}
Evaluating the inner integral on the right hand side, $\int_t^{\infty}\frac{1}{r^{\alpha}}\frac{dr}{r} = \tfrac{1}{\alpha}\tfrac{1}{t^{\alpha}}$, we conclude that the right hand side of this final inequality equals a constant multiple of $\mathbb{W}_2(\nu)(x)$.  The lemma follows.
\end{proof}

\section{Riesz systems}\label{riesz}

Throughout this appendix, fix a non-trivial locally finite measure $\mu$.  Recall that

$$\Phi_A^{\mu}(Q) = \Bigl\{\psi\in \Lip_0(B(x_Q, A\ell(Q))) : \|\psi\|_{\Lip}\leq\frac{1}{\ell(Q)}, \int_{\mathbb{R}^d}\psi d\mu=0\Bigl\}.$$

We shall prove that there is a constant $C=C(A)>0$ such that for each $f\in L^2(\mu)$, and arbitrary choices of $\psi_Q\in \Phi_A^{\mu}(Q)$, we have that
$$\sum_{Q\in \mathcal{D}}\frac{|\langle f, \psi_Q\rangle_{\mu}|^2}{\mu(B(x_{Q}, 3A\ell(Q)))}\leq C\|f\|_{L^2(\mu)}^2.
$$
Here, as well as elsewhere in this appendix, the sum over the dyadic cubes is to be taken over those cubes with $\mu(B(x_{Q}, 3A\ell(Q)))>0$.

We shall prove this inequality by verifying its equivalent dual inequality:  there is a constant $C=C(A)>0$ such that for each non-negative sequence $(a_Q)_Q \in \ell^2(\mathcal{D})$, and for every choice of $\psi_Q\in \Phi_A^{\mu}(Q)$, we have that
$$\Bigl\|\sum_{Q\in \mathcal{D}}\frac{a_Q\psi_Q}{\sqrt{\mu(B(x_{Q}, 3A\ell(Q)))}}\Bigl\|_{L^2(\mu)}^2\leq C\|a\|_{\ell^2}^2.
$$

It will be convenient to set $\rho_Q = \mu(B(x_Q, 3A\ell(Q)))$.
We begin the proof with a few preparatory estimates. For each $Q\in \mathcal{D}$, choose $\psi_Q\in \Phi_A^{\mu}(Q)$.  Then
$$\|\psi_Q\|_{\infty}\leq\|\psi_{Q}\|_{\Lip}\cdot \diam(\supp(\psi_Q))\leq CA.
$$
Thus,
$$\|\psi_Q\|_{L^1(\mu)} \leq CA\mu(B(x_Q, A\ell(Q)).
$$
Notice that if $Q', Q''\in \mathcal{D}$ with $\ell(Q')\leq \ell(Q'')$, then the oscillation of $\psi_{Q''}$ on $B(x_{Q'},A\ell(Q'))$ is bounded by $\tfrac{A\ell(Q')}{\ell(Q'')}$.  Thus
$$\frac{|\langle \psi_{Q'}, \psi_{Q''}\rangle_{\mu}|}{\sqrt{\rho_{Q'}\rho_{Q''}}} \leq C(A)\frac{\ell(Q')}{\ell(Q'')} \sqrt{\frac{\mu(B(x_{Q'}, A\ell(Q')))}{\rho_{Q''}}}.$$
Also note that
$|\langle \psi_{Q'}, \psi_{Q''}\rangle_{\mu}|=0$ if $B(x_{Q'}, A\ell(Q'))\cap B(x_{Q''}, A\ell(Q''))=\varnothing$.

For the remainder of this proof, all sums over cubes will be taken over the dyadic lattice $\mathcal{D}$, so we shall not write this explicitly.  Now, let $(a_Q)_{Q\in \mathcal{D}}\in \ell^2(\mathcal{D})$.  Then
$$\Bigl\|\sum_{Q}\frac{a_Q\psi_Q}{\sqrt{\rho_{Q}}}\Bigl\|_{L^2(\mu)}^2
\leq 2\sum_{Q', Q'': \ell(Q')\leq \ell(Q'')}|a_{Q'}||a_{Q''}| \frac{|\langle \psi_{Q'}, \psi_{Q''}\rangle_{\mu}|}{\sqrt{\rho_{Q'}\rho_{Q''}}}.
$$

Appealing to our previous estimates,  Cauchy's inequality yields that $|a_{Q'}||a_{Q''}| \tfrac{|\langle \psi_{Q'}, \psi_{Q''}\rangle_{\mu}|}{\sqrt{\rho_{Q'}\rho_{Q''}}}$ is bounded by 
$$C(A)\Bigl[\frac{|a_{Q'}|^2}{2}  \frac{\ell(Q')}{\ell(Q'')}+ \frac{|a_{Q''}|^2}{2}\frac{\ell(Q')}{\ell(Q'')} \frac{\mu(B(x_{Q'}, A\ell(Q')))}{\rho_{Q''}}\Bigl].
$$
Thus, it suffices to estimate two sums:
$$I= \sum_{\substack{Q', Q'': \ell(Q')\leq \ell(Q'')\\ B(x_{Q'}, A\ell(Q'))\cap B(x_{Q''}, A\ell(Q''))\neq\varnothing}}\!\!\!\!\!\!\!|a_{Q'}|^2 \frac{\ell(Q')}{\ell(Q'')},$$ and
$$II= \sum_{\substack{Q', Q'': \ell(Q')\leq\ell(Q'')\\ B(x_{Q'}, A\ell(Q'))\cap B(x_{Q''}, A\ell(Q''))\neq\varnothing}}\!\!\!\!\!\!\!|a_{Q''}|^2 \frac{\ell(Q')}{\ell(Q'')} \frac{\mu(B(x_{Q'}, A\ell(Q')))}{\rho_{Q''}}.$$
Fix $Q'$ and $k\in \mathbb{Z}_+$.  There are at most $C(A)$ cubes $Q''$ with $\ell(Q'')=2^k\ell(Q')$ satisfying $ B(x_{Q'}, A\ell(Q'))\cap B(x_{Q''}, A\ell(Q''))\neq\varnothing$. Thus
$$I = \sum_{Q'}|a_{Q'}|^2\!\!\!\!\!\!\!\!\!\!\!\!\!\! \sum_{\substack{Q'': \ell(Q')\leq \ell(Q'')\\ B(x_{Q'}, A\ell(Q'))\cap B(x_{Q''}, A\ell(Q''))\neq\varnothing}}\!\!\!\!\!\!\!\!\!\!\frac{\ell(Q')}{\ell(Q'')} \leq C(A) \sum_{Q'}|a_{Q'}|^2\sum_{k\in \mathbb{Z}_+}2^{-k},
$$
which is at most $C(A)\sum_{Q'}|a_{Q'}|^2$.
For $II$, write
$$II = \sum_{Q''}|a_{Q''}|^2 \sum_{k\in \mathbb{Z}_+}2^{-k}\!\!\!\!\!\!\! \sum_{\substack{Q':\ell(Q') = 2^{-k}\ell(Q''))\\ B(x_{Q'}, A\ell(Q'))\cap B(x_{Q''}, A\ell(Q''))\neq\varnothing}} \frac{\mu(B(x_{Q'}, A\ell(Q')))}{\rho_{Q''}}.
$$
With $k\in \mathbb{Z}_+$ fixed, the inner sum can be written as
\begin{equation}\label{innerterm}\frac{1}{\rho_{Q''}}\int\limits_{\mathbb{R}^d}\!\!\!\sum_{\substack{Q':\ell(Q') = 2^{-k}\ell(Q''))\\ B(x_{Q'}, A\ell(Q'))\cap B(x_{Q''}, A\ell(Q''))\neq\varnothing}}\!\!\!\!\!\!\!\!\!\!\!\!\!\!\!\!\!\!\!\!\!\!\!\chi\ci{B(x_{Q'}, A\ell(Q'))}(y)d\mu(y).
\end{equation}
Note that if $B(x_{Q'}, A\ell(Q'))\cap B(x_{Q''}, A\ell(Q''))\neq\varnothing$, then $B(x_{Q'}, A\ell(Q'))\subset B(x_{Q''}, 3A\ell(Q''))$.  Thus, the domain of integration in the above integral  may be restricted to $B(x_{Q''}, 3A\ell(Q''))$.   On the other hand, any point $y\in B(x_{Q''}, 3A\ell(Q''))$ lies in at most $C(A)$ distinct balls $B(x_{Q'}, A\ell(Q'))$ corresponding to the cubes $Q'\in \mathcal{D}$ with $\ell(Q') = 2^{-k}\ell(Q'')$.  Consequently, the integrand is bounded by $C(A)$.  Therefore, the quantity in display (\ref{innerterm}) is bounded by $\tfrac{1}{\rho_{Q''}}C(A)\rho_{Q''}\leq C(A)$.  This yields $II\leq C(A)\sum_{Q''}|a_{Q''}|^2$, completing the proof.

\section{The Representation of the Fractional Laplacian: The proof of Lemma \ref{sharmon}}

In this section we shall prove Lemma \ref{sharmon}.  Set $K(x) = \tfrac{x}{|x|^{s+1}}, x\in \mathbb{R}^d$ with $s\in (d-1,d)$.  Fix a $\Lambda$-nice measure $\mu$ with $\dist(0, \supp(\mu))=1$.  The proof is a rather tiresome approximation argument, based around the  formula (see, for example \cite{Lan},\cite{ENV}): if $g\in C^{\infty}_{0}(\mathbb{R}^d)$,
$$
\nabla g(x) = b P.V. \int_{\mathbb{R}^d}\frac{T(gm_d)(x)- T(gm_d)(y)}{|x-y|^{2d+1-s}}dm_d(y),
$$
where $b\in \mathbb{R}\backslash \{0\}$, and $T(gm_d)(x) = \int_{\mathbb{R}^d}K(x-y)d\mu(y)$.

Fix a mollifier $\psi\in C^{\infty}_0(B(0,1))$ satisfying $\int_{\mathbb{R}^d} \psi dm_d=1$.  For $\rho>0$, set $\psi_{\rho}=\rho^{-n}\psi(\rho\,\cdot\,)$.  If $\rho\in (0,\tfrac{1}{4})$ and $N>0$ are given, then set $\mu_{\rho, N}  =\psi_\rho*(\mu\chi_{B(0,N)})$.  Notice that $\mu_{\rho, N}$ is a measure with $C^{\infty}_0(\mathbb{R}^d)$ density $g_{\rho, N}$ with respect to $m_d$, and since $\dist(0, \supp(\mu))=1$, $\nabla g_{\rho, N}(0)=0$.

Thus
$$0 = P.V. \int_{\mathbb{R}^d} \frac{T(\mu_{\rho, N})(0) - T(\mu_{\rho, N})(x)}{|x|^{2d+1-s}}dm_d(x).
$$

Our strategy is clear:  let $N\rightarrow \infty$, $\rho \rightarrow 0$, and argue that the right hand side converges to
$$P.V. \int_{\mathbb{R}^d} \frac{\overline{T}_{\mu}(1)(x)  - \overline{T}_{\mu}(1)(0)}{|x|^{2d+1-s}} dm_d(x).
$$

In order to pass to the limit in this fashion, we shall require some preparatory estimates.  Suppose that $\nu$ is a $\Lambda'$-nice measure with $\dist(0, \supp(\nu))\geq \tfrac{1}{2}$.  Notice that, for ever $x'\in B(0, \tfrac{1}{4})$, and $m_d$-almost every $x\in \mathbb{R}^d$,
\begin{equation}\label{diffeq}\T1_{\nu}(1)(x) - \T1_{\nu}(1)(x') = \int_{\mathbb{R}^d} [K(y-x)-K(y-x')]d\nu(y).
\end{equation}
Also note that $|x-\cdot|^{-\beta}\in L^1(\nu)$ whenever $x\in B(0, \tfrac{1}{4})$ and $\beta>s$.  In particular, this implies that for any multi-index $\gamma$ with $|\gamma|\geq 1$,
$$|D^{\gamma}\T1_{\mu}(1)(x)| = \Bigl|\int_{\mathbb{R}^d}[D^{\gamma}K(y-x)]d\nu(y)\Bigl|\leq C(\gamma,\Lambda'),
$$
for any $x\in B(0,\tfrac{1}{4})$.  In concert with the elementary inequality
$$\Bigl|\int_{\partial B(0,1)}[f(ry) - f(0)] dS(y)\Bigl|\leq Cr^2\sup_{B(0,r)}|\Delta f|,
$$
valid for $f\in C^2(\overline{B(0,r)})$ (here $S$ is the surface measure on the unit sphere), we see that if $r<\tfrac{1}{4}$ and $\tau>0$
\begin{equation}\begin{split}\label{smoothnearzero}\Bigl|&\int_{B(0,r)\backslash B(0,\tau)} \frac{\T1_{\nu}(0) - \T1_{\nu}(x)}{|x|^{2d+1-s}}dm_d(x)\Bigl|\\
&=c\Bigl|\int_{\tau}^r\frac{1}{t^{2d+1-s}}\int_{\partial B(0,1)}[\T1_{\nu}(1)(0)-\T1_{\nu}(1)(x)] dS(x) t^{d-1}dt\Bigl|\\
&\leq C(\Lambda')r^{1+s-d}.
\end{split}\end{equation}

Next, note that if $A>1$, then
\begin{equation}\label{avenotbig}
\int_{B(0,A)}|\T1_{\nu}(1)(0)-\T1_{\nu}(1)(x)|dm_d(x)\leq C(\Lambda')A^d\log(e+A).
\end{equation}
To see this, note that the left hand side is dominated by the sum
\begin{equation}\begin{split}\nonumber\int_{B(0,A)}&\Bigl|\int_{B(0,2A)}[K(y)-K(y-x)]d\nu(y)\Bigl|dm_d(x) \\
&+ \int_{B(0,A)}\Bigl|\int_{\mathbb{R}^d\backslash B(0,2A)}[K(y)-K(y-x)]d\nu(y)\Bigl|dm_d(x)
\end{split}\end{equation}
Let us call the two terms appearing here $I$ and $II$.

We estimate term $I$ by $\int_{B(0,A)}\int_{B(0,2A)}[|K(y)|+|K(y-x)|] d\mu(y)d\mu(x)$, which is bounded by
$$\int_{B(0,A)}\int_{B(0,2A)}\frac{1}{|y-x|^s}d\nu(y)dm_d(x) + CA^d\int_{B(0,A)\backslash B(0, \tfrac{1}{4})}\frac{1}{|y|^s}d\nu(y).
$$
But the sum of these two integrals is easily seen to at most $C\nu(B(0, 2A)) A^{d-s}+ CA^d\log(e+A)\leq CA^d\log(e+A)$.  For term $II$, merely note that
$$\int_{\mathbb{R}^d\backslash B(0,2A)}|K(y)-K(y-x)|d\nu(y)\leq \int_{\mathbb{R}^d\backslash B(0,2A)}\frac{|x|}{|y|^{s+1}}d\nu(y)\leq C(\Lambda')
$$
for $x\in B(0, A)$.  Bringing these estimates together yields the inequality (\ref{avenotbig}).

Integrating the estimate (\ref{avenotbig}) yields that

\begin{equation}\begin{split}\label{tailsmallest}\int_{\mathbb{R}^d\backslash B(0,A)}&\frac{|\T1_{\nu}(1)(0)-\T1_{\nu}(1)(x)|}{|x|^{2d+1-s}}dm_d(x)\\
&\leq C(\Lambda')\int_A^{\infty}\frac{1}{Q^{2d+1-s}} Q^d\log(e+Q) \frac{dQ}{Q}\leq \frac{C(\Lambda')}{A^{d-s}}.
\end{split}\end{equation}

We are now ready to proceed with the limiting procedure.  Notice that if $\rho<\tfrac{1}{4}$ and $N>0$, then $\mu_{N, \rho}$ is a $\Lambda'=C(\Lambda,d)$-nice measure, with $\dist(0, \supp(\mu_{N, \rho}))\geq \tfrac{1}{2}$.

Let $\eps>0$.  Then (\ref{smoothnearzero}) yields the existence of $r>0$ such that
\begin{equation}\begin{split}\Bigl|P.V. \int_{B(0,r)} &\frac{\T1_{\mu_{N,\rho}}(0) - \T1_{\mu_{N,\rho}}(x)}{|x|^{2d+1-s}}dm_d(x)\Bigl|\\
& + \Bigl|P.V. \int_{B(0,r)} \frac{\T1_{\mu}(0) - \T1_{\mu}(x)}{|x|^{2d+1-s}}dm_d(x)\Bigl|< \eps
\end{split}\end{equation}
for all $\rho \in (0, \tfrac{1}{4})$ and $N>1$.  In addition, the estimate (\ref{avenotbig}) ensures that there exists $A>0$ such that
\begin{equation}\begin{split}\int_{\mathbb{R}^d\backslash B(0,A)}&\frac{|\T1_{\mu_{\rho,N}}(1)(0)-\T1_{\mu_{N,\rho}}(1)(x)|}{|x|^{2d+1-s}}dm_d(x)\\
&+ \int_{\mathbb{R}^d\backslash A}\frac{|\T1_{\mu}(1)(0)-\T1_{\mu}(1)(x)|}{|x|^{2d+1-s}}dm_d(x)\leq \eps
\end{split}\end{equation}
for all $\rho \in (0, \tfrac{1}{4})$ and $N>1$.

Next, note that if $N>2A$, and $x\in B(0,A)$, then
\begin{equation}\begin{split}\nonumber|[\T1_{\mu}&(1)(x) - \T1_{\mu}(1)(0)] -[\T1_{\chi_{B(0,N)}\mu}(1)(x)-\T1_{\chi_{B(0,N)}\mu}(1)(0)]|\\
&=\Bigl|\int_{\mathbb{R}^d\backslash B(0,N)}[K(y-x)-K(y)]d\mu(y)\Bigl|\leq C(\Lambda)\frac{A}{N}.
\end{split}\end{equation}

So now fix $N$ so that
\begin{equation}\begin{split}\nonumber\int_{B(0,A)}&\Bigl|[\T1_{\mu}(1)(x) - \T1_{\mu}(1)(0)] \\
&-[\T1_{\chi_{B(0,N)}\mu}(1)(x)-\T1_{\chi_{B(0,N)}\mu}(1)(0)]\Bigl|dm_d(x)\leq \eps r^{2d+1-s}.
\end{split}\end{equation}
But, with $r, A$ and $N$ fixed, note that
\begin{equation}\begin{split}\nonumber\int_{B(0,A)}&\bigl|[\T1_{\mu_{\rho,N}}(1)(x) - \T1_{\mu_{\rho,N}}(1)(0)] \\
&-[\T1_{\chi_{B(0,N)}\mu}(1)(x)-\T1_{\chi_{B(0,N)}\mu}(1)(0)]\bigl|dm_d(x)\\
& =\int_{B(0,A)}\Bigl|\rho_{\rho}*\Bigl[\int_{B(0,N)}K(y-\cdot)d\mu(y)\Bigl](x) - \int_{B(0,N)}K(y-x)d\mu(y)\\
& - \rho_{\rho}*\Bigl[\int_{B(0,N)}K(y-\cdot)d\mu(y)\Bigl](0) + \int_{B(0,N)}K(y)d\mu(y)\Bigl|dm_d(x)
\end{split}\end{equation}

But notice that, as $\rho\rightarrow 0$,
$$\int_{B(0,A)}\Bigl|\rho_{\rho}*\Bigl[\int_{B(0,N)}K(y-\cdot)d\mu(y)\Bigl]-\int_{B(0,N)}K(y-x)d\mu(y)\Bigl|dm_d(x)
$$
converges to $0$, by standard theory $\bigl(\int_{B(0,N)}K(y-\cdot)d\mu(y)$ lies in $L^1(\chi_{B(0,A)}m_d)\bigl)$, while clearly
$$ \rho_{\rho}*\Bigl[\int_{B(0,N)}K(y-\cdot)d\mu(y)\Bigl](0) - \int_{B(0,N)}K(y)d\mu(y)
$$
converges to zero with $\rho$.

Finally, it remains to notice that  $$T(\mu_{\rho, N})(0) - T(\mu_{\rho, N})(x) =  \T1_{\mu_{\rho, N}}(1)(x) - \T1_{\mu_{\rho, N}}(1)(0),$$ and the triangle inequality yields
$$\Bigl|P.V. \int_{\mathbb{R}^d} \frac{\overline{T}_{\mu}(1)(0)  - \overline{T}_{\mu}(1)(x)}{|x|^{2d+1-s}} dm_d(x)\Bigl|\leq 5\eps.
$$

 \end{document}